\documentclass[12pt]{amsart}
\usepackage{amsmath}
\usepackage[utf8]{inputenc}
\usepackage{amsthm}
\usepackage{amsfonts}
\usepackage{amssymb}
\usepackage{mathrsfs}
\usepackage{fullpage}
\usepackage[english]{babel}
\usepackage{tikz-cd}
\usepackage{hyperref}
\usepackage{thmtools}

\usepackage{graphicx}
\usepackage{float}
\restylefloat{table}
\restylefloat{figure}

\usepackage{amssymb}
\usepackage{hyperref}
\date{29th September 2020}

\newcommand\n{\mathbf n}
\newcommand{\hooklongrightarrow}{\lhook\joinrel\longrightarrow}

\newtheorem{theorem}{Theorem}
\newtheorem{proposition}{Proposition}
\newtheorem{corollary}{Corollary}
\newtheorem{lemma}{Lemma}
\theoremstyle{remark}

\title[Determinant of representations of $\mathbb{Z}_r\wr S_n$ ]{On the determinant of representations of generalized symmetric groups $\mathbb{Z}_r\wr S_n$}
\author{Amrutha P}\address{Indian Institute of Science Education and Research Thiruvananthapuram, Thiruvananthapuram}\email{amruthap15@iisertvm.ac.in}
\author{T. Geetha}\address{Indian Institute of Science Education and Research Thiruvananthapuram, Thiruvananthapuram}\email{tgeetha@iisertvm.ac.in}
\begin{document}
\maketitle
\begin{abstract}
In this paper we study the determinant of irreducible representations of the generalized symmetric groups $\mathbb{Z}_r \wr S_n$. We give an explicit formula to compute the determinant
of an irreducible representation of $\mathbb{Z}_r \wr S_n$. Recently, several authors have characterized and counted the number of irreducible representations of a given finite group with nontrivial determinant. Motivated by these results, for given integer $n$, $r$ an odd prime number and $\zeta$ a nontrivial multiplicative character of $\mathbb{Z}_r \wr S_n$ with $n<r$, we obtain an explicit formula to compute $N_{\zeta}(n)$, the number of irreducible representations of $\mathbb{Z}_r \wr S_n$ whose determinant is $\zeta$.
  \end{abstract}
 
\section{Introduction}
The main aim of this paper is to count and characterize the number of irreducible representations of the generalized symmetric group whose composition with the determinant map being the given multiplicative character of $\mathbb{Z}_r \wr S_n$. In a recent paper \cite{amri}, the authors characterize the irreducible representations of $S_n$ having nontrivial determinant by characterizing the corresponding integer partitions $\lambda$ of $n$. They gave a closed formula to count the number of irreducible representations of $S_n$ with nontrivial determinant. In \cite{spallone}, the authors extended this result to all the irreducible finite coxeter groups $W$.  Both in \cite{amri} and \cite{spallone}, the authors used a simple way to read off the determinant of a representation from its $2$-core tower. For given an integer $n$ and a multiplicative character $\zeta$, they gave a closed formula to count the number of irreducible representations with nontrivial determinant $\zeta$ and gave a characterization of all the representations whose determinant is $\zeta$. The generalized symmetric groups are the wreath product group of the cyclic group with the symmetric group which is a natural group theoretic construction with many interesting applications. Some interesting special cases of these groups are the symmetric group and the hyper-octahedral group.  

All the results here and in \cite{amri} and \cite{spallone} have their genesis in \cite{mac}. In \cite{mac}, Macdonald developed a combinatorics for partitions and gave a closed formula to count the number of odd-dimensional Specht modules for the symmetric groups. This number happens to be the product of the powers of $2$ in the binary expansion of $n$ and was obtained by characterizing the $2$-core tower of the odd partitions. In the case of symmetric group, a representation $\rho$ of $S_n$ is said to be a \textit{chiral representation} if the composition of $\rho$ with the determinant map is nontrivial. As the irreducible representations of the symmetric, hyper-octahedral and generalized symmetric group are indexed with partitions, bi-partitions and multipartitions respectively, one can call a partition (resp. a multipartition) is a \textit{chiral partition} if the corresponding representation is chiral.  

One natural question arises out of this is to ask for the count of number of chiral partitions of a given integer $n$. A very usual method to deal with understanding a problem involving the symmetric group is to treat all the $n$ simultaneously. A particular case of this approach is the problem of counting the number of chiral partitions for a given integer $n$ which was first considered by L. Solomon and later posed by Stanley in \cite{Stanley}. A closed formula for the number of chiral partitions was obtained in \cite{amri}, for the case of $S_n$  and in \cite{spallone}, for the case of hyper-octahedral group. In fact, the authors in \cite{spallone} gave much more than just counting the number of chiral partitions by differentiating the two sets of multicompositions with two parts whose determinant is nontrivial. Since in the case of hyper-octahedral group, there are four characters involved such that two of them are trivial and two of them are nontrivial and the process of counting involves counting the number of irreducibles whose determinant is coming from the different nontrivial characters. Our construction in this paper generalizes the formulas in \cite{mac} and certain results from \cite{amri}, \cite{spallone} to the case of $\mathbb{Z}_r \wr S_n$. Throughout this paper, we denote $\mathbb{Z}_r \wr S_n$ with $G(n,r)$.

The main results and the structure of the paper is as follows: In Section 2, we define some preliminary notations and prove few combinatorial results which will be used in the rest of the paper. In Section 3, we set up the definitions and recall some background material for the representation theory of the generalized symmetric groups $G(n,r)$. We give a closed formula to count the number of inequivalent odd degree irreducible representations of $G(n,r)$ and also a formula for the number $m_p(G(n,r))$, the number of inequivalent irreducible complex representations of $G(n,r)$ whose degree is relatively prime to a prime number $p$. For this we uses an approach similar to that of the one used by Macdonald in \cite{mac_1}.

In Section 4, we describe a formula to compute the determinant of the irreducible representations of $G(n,r)$. From the classification of the irreducible representations of $G(n,r)$, for given a primitive $r$th roots of unity $\zeta$, it is easy to note that the possible values of the determinant of an irreducible representation of $G(n,r)$ are given by $\pm \zeta^s$ where $1\leq s\leq r$. In the process of counting and characterizing the multipartitions with a given determinant, the interesting part is that the compositions of $n$ is adequate for this purpose rather than using all of the multipartitions of $n$. For given a composition $a=(a_1,\dots,a_r)$ of $n$, we denote $N_{\zeta^s}(a)$ (resp. $N_{-\zeta^s}(a)$) as the number of multipartitions of $n$ obtained from the composition $a$ whose determinant is ${\zeta^s}$ (resp. ${-\zeta^s}$). 

Investigating further on this in Section 5, we obtain a characterization of the multipartitions of $n$ into $r$ parts whose determinant is $\pm 1$ and a positive power of $\zeta$ by characterizing the underlying  composition $(a_1,\dots,a_r)$ of $n$ involving in the given multipartition. The advantage of this characterization is that irrespective of the component parts of the multipartitions in the determinant formula, it uses only the conditions on the multinomial co-efficients involving in the parts of the underlying composition of $n$. Atlast in the Table \ref{table1}, for given an integer $n$ we provide a list of few special compositions which is worth of mentioning separately, whose determinant of the corresponding irreducible representation could be easily calculated from the characterizations obtained in this section. Next in Section 6, we obtain the following inequalities:
\begin{enumerate}
\item $N_{\zeta ^s}(n)$ (resp. $N_{-\zeta ^s}(n)$) are all equal for $1 \leq s  <r$
\item if $n < r$, $N_{\zeta ^s}(n)$ (resp. $N_{-\zeta ^s}(n)$) are all equal for $1 \leq s \leq r$
\item if $n$ is not a multiple of $r$, $N_{\zeta ^s}(n) \leq N_{1}(n)$ and $N_{-\zeta ^s}(n) \leq N_{-1}(n)$. 
\end{enumerate}
Finally in Section 7, we first present a formula to count the number of multipartitions of $n$ obtained from a composition $a=(a_1,\dots,a_r)$ of $n$ upon permutating its components, such that its determinant is positive (or negative). This result is the starting point of our computation for our main result. Using this we calculate $N_{\zeta ^s}(a)$ (and $N_{-\zeta ^s}(a)$) for certain special compositions $a$ of $n$. For given a composition $a$ of $n$ with $n<r$, $r$ is an odd prime and $1 \leq s \leq r$, the main result provides us with a closed formula to compute $N_{\zeta ^s}(a)$. In Appendix A, we give an alternate proof for the determinant formula in the Section 3 and in Appendix B, we give some interesting tables and graphs computed using Sage Math which gives the values for $N_{\zeta ^s}(n)$ and the graph of $N_{\zeta ^s}(n)$ for small values of $n$ and $r$.
\section{Preliminaries}
In this section we will introduce some notations and we prove some combinatorial results which will be used in section 5.

Throughout this paper we will consider only the finite-dimensional complex representations. Also we mean by a multiplicative character of a group $G$ to be the group homomorphism $G\rightarrow \mathbb{C}^{\times}$. For a representation $\rho$ of $G$, we will denote its character by $\chi_{\rho}$. For a representation $(\rho,V)$ of $G$ we denote $det \rho$ to be the composition of $\rho:G \rightarrow GL(V)$ with the determinant map, where $GL(V)$ is the general linear group. Clearly, $det\rho$ is a multiplicative character of $G$. Given a group $G$, let we denote the derived (commutator) subgroup of $G$ by $D(G)$ and the abelianization of $G$ given by $G/D(G)$ with $G_{ab}$. For given two groups $G$ and $S$ with representations $\rho_G$ and $\rho_S$ respectively, for the external tensor product representation of $G\times S$, we write $\rho_G\boxtimes \rho_S$. 

A representation $\rho$ of $G$ is said to be \textit{chiral} if $det \rho$ is nontrivial that is, the underlying group action is chiral in the sense of [\cite{Olsson}, Section 2].  In the case of $S_n$, $\rho$ is chiral iff $det \rho$ is the sign character. Let $n\geq2$ and $s_1=(12)\in S_n$. For $\lambda$ a partition of $n$, take 
\begin{equation*}
g_{\lambda}=\dfrac{f_{\lambda}-\chi_{\lambda}(s_1)}{2}. 
\end{equation*}
From [\cite{amri},Theorem 8], a partition $\lambda$ is chiral iff $g_{\lambda}$ is odd.
Given any non-negative integer $n$ and a prime $r$, we can write 
\begin{equation}
n=\alpha_0+\alpha_1 r^{k_1}+\alpha_2 r^{k_2}+\dots+\alpha_l r^{k_l}
\end{equation}
with $0\leq \alpha_i\leq r-1$. When $r=2$, one can call this the binary expansion of $n$ and we denote $bin(n)$ to be the set of powers of $2$ in the binary expansion of $n$ and $ord(n)$ for the highest power of $2$ dividing $n$.

Let us recall a result by E. Lucas from \cite{ELucas}.
\begin{theorem}\label{t1}
Let $m,n$ be two non-negative integers. If $r$ is a prime number and 
\begin{align*}
m&=a_kr^k+\dots +a_1r+a_0,\\
n&=b_kr^k+\dots +b_1r+b_0
\end{align*}
are the base $r$ expansions of $m$ and $n$ respectively. Then the following congruence relation holds:
\begin{equation}
{m\choose n} = \prod\limits_{i=0}^k {a_i\choose b_i} \mod r.
\end{equation}
\end{theorem}
Here we use the convention that ${m\choose n}=0$ if $m<n$.

That is, a binomial co-efficient is divisible by a prime $r$ iff atleast one of the base $r$ digits of $n$ is greater than the corresponding digit of $m$.

As an immediate consequence of this when $r=2$, we have the following:
\begin{lemma}\label{l1}
For $m,n$ positive integers, the binomial co-efficient ${m\choose n}$ is always even, iff atleast one of the base $2$ digits of $n$ is greater than the corresponding digit of $m$.
\end{lemma}
The binomial coefficients can be generalized to multinomial coefficients as follows:
\begin{equation*}
{n \choose {a_1,\dots,a_r}} = \dfrac{n!}{a_1!a_2!\dots a_r!} = {a_1\choose a_1} {{a_1+a_2}\choose a_2}\dots{{a_1+\dots +a_r}\choose{a_r} }.
\end{equation*}
\section{Irreducible representations of $G(n,r)$}
Let $n\in \mathbb{N}$. A {\it composition} $a$ of $n$ is a finite sequence of nonnegative integers whose sum is $n$. Let $C(n)$ denotes the set of all compositions of $n$.  A {\it partition} $\lambda$ of $n$ is a composition whose entries are non increasing. That is, $\lambda = (a_1,a_2, \dots , a_m)$
satisfying $a_1\geq a_2\geq \dots \geq a_m >0$ with $\sum_{i=1}^ma_i =n$. We call $a_i$'s to be the \textit{parts} of the partition $\lambda$. We denote the sum of the parts of a composition $a$ by $|a|$.  Let we denote the set of all partitions of $n$ by $P(n)$. We will denote a composition by $a \models n$ and a partition by $\lambda \vdash n$.

\indent Let $S_n$ denotes the symmetric group on $n$ letters $\{1,2,\dots ,n\}$. Let $\mathbb{Z}_r$ be the cyclic group of order $r$. Then $S_n$ acts on $\mathbb{Z}_r^n=\mathbb{Z}_r\times \mathbb{Z}_r\times \dots \times \mathbb{Z}_r$ ($n$ factors) by permuting the co-ordinates. The \textit{generalized symmetric group} $G(n,r)$ is the semi-direct product $\mathbb{Z}_r^n \rtimes S_n$ of $\mathbb{Z}_r^n$ by $S_n$. The group $G(n,r)$ is also called the wreath product of $\mathbb{Z}_r$ by $S_n$ and is denoted by $\mathbb{Z}_r\wr S_n$. Furthermore, in the group $G(n,r)$ if we take $r=1$, we have the symmetric group and if we take $r=2$, we get the hyperoctahedral group $B_n$ respectively.  The order of $G(n,r)$ is $r^nn!$.

\indent A \textit{multicomposition} of an integer $n$ is a sequence $a=(a_1,a_2,\dots,a_m)$ of compositions $a_i=(a_1^{(i)},a_2^{(i)},\dots,a_{s_i}^{(i)})$ such that $|a|=\sum_{i=1}^m |a_i|=n$. A \textit{multipartition} is a multicomposition in which each of it parts $a_i$  are also partitions. We will denote the set of all multicompositions and multipartitions of $n$ into $r$ parts respectively by $C(n,r)$ and $P(n,r)$. The representation theory of $G(n,r)$ is well-known and see \cite{JamesKerber} and \cite{Kerber} for more details. 

\indent Each $a=(a_1,a_2,\dots,a_r)\in C(n)$ can be represented by a \textit{Young diagram} $[a]$ which is a left-justified array of boxes with $a_i$ boxes in the $i$th row. For $a\in C(n)$, a \textit{row (column) standard Young tableau} of shape $a$ is the filling of boxes in the Young diagram of $a$ with entries from $\{1,2,\dots,n\}$ increasing along rows (columns respectively). A \textit{standard Young tableau} is the one which is both row and column standard. The symmetric group acts on the set of all Young tableaux by permuting the entries of the tableaux. 

 For $a\in C(n)$, define the unique tableau of shape $a$, $T_{a}$ which is a row standard Young tableau in which the integers $1,2,\dots,n$ entered in increasing order from left to right along the rows of $a$. The \textit{Young subgroup} $S_a=S_{a_1}\times \dots \times S_{a_r}$ of $S_n$ is the row stabilizer of $T_{a}$. For $a \in C(n)$, $a = (a_1,a_2,\dots,a_r)$, define $t_0=0$ and $t_i=\sum\limits _{j=1}^ia_j,\,i=1,\dots,n$. Let the \textit{generalized symmetric group} on $a_i$ letters $\{t_{i-1}+1,\dots,t_i\},\, i=1,\dots,m$ be denoted by $G(a_i,r)$. Let $G(0,r)$ be the trivial subgroup of $G(n,r)$. The group $G(a_1,r)\times\dots\times G(a_r,r)$ is called the \textit{generalized Young subgroup} determined by $a$ and is denoted by $G(a,r)$. 

The inequivalent irreducible complex representations of $G(n,r)$ is indexed with the multipartitions $\lambda \in P(n,r)$ and we denote the corresponding irreducible representation by $\rho_{\lambda}$.  Let $\zeta^i$ denote the character of $\mathbb{Z}_r^n$, whose restriction to  each factor $\mathbb{Z}_r$  sends the identity to $\zeta^i$, where $\zeta$ is a primitive $r$th roots of unity. As it is $S_n$- invariant, it extends to a multiplicative character $\zeta^i$ in $G(n,r)$.  

Let $\lambda\vdash n$. Consider the extensions of the representation $\rho_\lambda$ of $S_n$ to $G(n,r)$ namely
\begin{center}
${\rho_\lambda}^k (x;w) = \zeta^k (x) \rho_{\lambda} (w)$ \hspace{1cm}  $0\leq k\leq r-1$
\end{center}
for $x \in \mathbb{Z}_r^n$ and $w\in S_n$. Let $\lambda=(\lambda_1,\dots,\lambda_r)\in P(n,r)$ with $\lambda_k \vdash a_k,1 \leq k \leq r$. Define
\begin{center}
$\rho _\lambda= {Ind}_{G(\lambda,r)} ^{G(n,r)} \boxtimes _{k=0}^{r-1} \rho_{\lambda_{k+1}}^i$.
\end{center}
Then $\{\rho _{\lambda} |\lambda \in P(n,r) \}$ is a complete set of representatives of the set of isomorphism classes of irreducible representation of $G(n,r)$. The dimension of the representation space $V_{\lambda} $ is given by
\begin{equation}\label{dimV}
f_{\lambda}= dimV_{\lambda} = f_{\lambda_1}\dots f_{\lambda_r}  {n \choose {a_1,\dots,a_r}}.
\end{equation}
We will denote the number of odd dimensional irreducible representations of $S_n$ (resp. $G(n,r)$) by $A(n)$ (resp. $A(n,r))$.
\begin{theorem}\label{t2}
The number of inequivalent irreducible (complex) representations of the generalized symmetric group $G(n,r)$ with odd degree is $r^{|bin(n)|} A(n)$.
\end{theorem}
\begin{proof}
From equation (\ref{dimV}), $f_{\lambda}$ is odd if and only if each $f_{\lambda_k}$'s for $1 \leq k \leq r$ and the multinomial coefficient ${n \choose {a_1,\dots,a_r}} $ are odd.
Now we claim that, given any $n$ and $r$, we have exactly $r^{|bin(n)|}$ odd multinomial coefficients in the expansion of ${(x_1+\dots+x_r)}^n$.

The multinomial coefficient can be written as
$${n \choose {a_1,\dots,a_r}} = {n \choose {a_1}} {n-a_1 \choose {a_2}}{n-(a_1+a_2) \choose {a_3}} \dots{a_r \choose {a_r}}.$$ Any ${n \choose {k}} $ is odd if and only if  whenever for every position for which $k$ has a one in binary, n also has a one. So, we have to choose $r$ numbers $a_1,\dots,a_r$ such that $bin(n) =bin(a_1) \cup bin(a_2) \cup\dots \cup bin(a_r)$, with $bin(a_1) \cap bin(a_2) \cap\dots\cap bin(a_r) = \emptyset $. Given any element in $bin(n)$ we have $r$ choices to place it, hence the total number of ways is equal to  $r^{|bin(n)|}$.

It is clear from the definition of $A(n)$ and the above description of the odd multinomial coefficients that given any odd multinomial coefficient we have exactly $A(n)$ terms such that the product $f_{\lambda_1}\dots f_{\lambda_r}$ is odd where each $\lambda_k  \vdash a_k$, $k\in \{1,\dots,r\}$.
\end{proof}
Let $G$ be a finite group and $p$ be a prime number, and let we denote $m_p(G)$ to be the number of inequivalent irreducible (complex) representations of $G$ whose degree is relatively prime to $p$. The formula for $m_p(G)$ when $G$ is the symmetric group and the hyperoctahedral group was derived in \cite{mac} and \cite{mac_1} respectively. Using the argument similar to that of in \cite{mac}, we have the following result.
\begin{theorem}\label{t3}
For any prime $p$ and an integer $r$, we have 
\begin{equation}
m_p(G(n,r)) =  \prod_{k\geq 0} [P(x)^{rp^k}]_{x^{\alpha_k}}
\end{equation}
where $\alpha_k$ is as in equation (1) and $P(x)$ is the partition generating function.
\end{theorem}
Then Theorem \ref{t2} can be obtained as a corollary to the above theorem with a special case of $p$ being $2$ or also, can be proved independently using the basic combinatorial techniques as shown.

\section{Determinant of $G(n,r)$}
Let us first briefly review the results from \cite{det} and \cite{spallone}. The following proposition gives the formula for the determinant of an induced representation of a finite group.
\begin{proposition}[29.2,\cite{det}]\label{p5}
Let $G$ be a finite group, $H$ is  a subgroup of $G$ and $\rho$ a representation of $H$. If $\pi = Ind_H^G \rho $, then 
\begin{equation*}
det \pi = det (\mathbb{C}[G/H])^{dim\rho} det{\ }\rho \circ ver _{G/H}, 
\end{equation*}
where $\mathbb{C}[G/H]=Ind_H^G 1$.
\end{proposition}
The key part in our calculation is the verlagerung map  $ver_{G/H} : G_{ab} \rightarrow H_{ab}$ between the abelianizations of $G$ and $H$. 
Let us briefly recall the $ver$ map:
 Let  $t:G/H\rightarrow G$ be a section of the canonical projection and let $g\in G$. For each $x\in G/H$, we have
$gt(x)=t(y)h_{x,g}$ for some $y \in G/H$ and $h_{x,g}\in H$.
Considering $G/H$ as a $G$-set, we have $y={^{g}x}$. Then
\begin{center}
$ver_{G/H} (g \mod D(G)) = \prod\limits_{x\in G/H} h_{x,g}\mod D(H)$.
\end{center}
Let $G=S_n$ and $H$ be the Young subgroup $S_{a_1} \times S_{a_2} \times\dots\times S_{a_r}$ with $a_1+a_2+\dots+a_r =n$. Then
$G_{ab}= S_n/A_n$ and
$H_{ab}= S_{a_1}/A_{a_1} \times S_{a_2}/A_{a_2} \times\dots\times S_{a_r}/A_{a_r}$. Now we generalize the Proposition 7 of \cite{spallone} as follows:
\begin{proposition}\label{p6}
 Let $a_1+a_2+\dots+a_r =n$. Then the map 
\begin{center}
$ver= ver _{S_n/S_{a_1} \times S_{a_2} \times\dots\times S_{a_r}} : S_n/A_n\rightarrow S_{a_1}/A_{a_1} \times S_{a_2}/A_{a_2} \times\dots\times S_{a_r}/A_{a_r}$
\end{center}
is given by
\begin{center}
$ver (\tau_n) = \Big(\tau_{a_1} ^ {{{n-2}\choose{{a_1-2},{a_2},\dots,{a_r}}}},\tau_{a_2} ^ {{{n-2}\choose{{a_1},{a_2-2},\dots,{a_r}}}},\dots,\tau_{a_r}^{{{n-2}\choose{{a_1},{a_2},\dots,{a_r-2}}}}\Big)$
\end{center}
where $\tau_n$ is any transposition in $S_n$ or trivial if $n<2$.
\end{proposition}
\begin{proof}
 Let $J_n=\{1,2,\dots,n\}$ and let we denote $\mathscr{P}(J_n)$ by set of all ordered multisets of subsets of $J_n$, whose union is $J_n$ and each set is of size $a_1,a_2,\dots,a_r$ respectively. We have an action of $S_n$ on $\mathscr{P}(J_n)$ by permuting the elements in the multiset. Let us fix $J = \{\{1,2,\dots,a_1\},\{a_1+1,\dots,a_1+a_2\},\dots,\{n-a_r,n-a_{r+1},\dots,n\}\} \in \mathscr{P}(J_n) $ and for $x\in \mathscr{P}(J_n)$, we denote $x_k$ the $k$th component in the multiset which is of size $a_k$. Then the map $S_n \rightarrow \mathscr{P}(J_n) $ defined by $g \rightarrow {^{g}J}$ descends to an isomorphism of $S_n/S_{a_1} \times S_{a_2} \times\dots\times S_{a_r}$ with $\mathscr{P}(J_n) $ as $S_n$-sets. Pick any section $t:\mathscr{P}(J_n)  \rightarrow S_n$, then
 \begin{equation*}
 {^{t(x)} J }= x, \text{ for all } x \in \mathscr{P}(J_n).
 \end{equation*}
 For concreteness, let us assume $\tau$ be the transposition $s_1=(12)$. Since
 \begin{equation*}
 h_{x,{s_1}}=t({^{s_1} x})^{-1} s_1t(x),
 \end{equation*}
 we have  $h_{x,{s_1}} h_{{^{s_1} x},{s_1}} =1$. Thus 
 \begin{align*}
ver (\tau_n)& = \prod_{x\in \mathscr{P}(J_n)} h_{x,{s_1}} \mod A_n\\
 &= \prod_{x\in \mathscr{P}(J_n)| {^{s_1}x}=x} h_{x,{s_1}} \mod A_n,
 \end{align*}
and ${^{s_1}x}=x$ if and only if $\{1,2\}$ lies entirely in one set in the multiset $x$. In this case $h_{x,{s_1}} =\tau_{a_i}$, for some transposition $\tau_{a_i}\in S_{a_i}$. There are ${{{n-2}\choose{{a_1},{a_2},\dots,{a_i-2},\dots,{a_r}}}}$ such elements of $\mathscr{P}(J_n)$ containing $\{1,2\}$ in its ith set. Therefore, 
 \begin{equation*}
ver (\tau_n) = \Big(\tau_{a_1} ^ {{{n-2}\choose{{a_1-2},{a_2},\dots,{a_r}}}},\tau_{a_2} ^ {{{n-2}\choose{{a_1},{a_2-2},\dots,{a_r}}}},\dots,\tau_{a_r}^{{{n-2}\choose{{a_1},{a_2},\dots,{a_r-2}}}}\Big).
\end{equation*}
\end{proof}
We can describe the multiplicative characters of $G(n,r)$ as follows: Denote $\varepsilon_i:\mathbb{Z}_r^n\longrightarrow\{\pm\zeta^i\}$ for $i\in\{0,1,\dots,r-1\}$ for the character whose restriction to each factor $\mathbb{Z}_r$ is $\zeta^i$. Being an $S_n$-invariant, we can extend this to a multiplicative character of $G(n,r)$. We may write $sgn$ for the composition of projection $G(n,r)\rightarrow S_n$ with the sign character of $S_n$. And we will denote $sgn^i=\varepsilon_i sgn$ for  $i\in\{1,\dots,r-1\}$. Thus the $2r$ multiplicative characters of $G(n,r)$ are $\{1,\varepsilon_1,\dots,\varepsilon_{r-1},sgn,sgn^1,\dots,sgn^{r-1}\}$.

Now let us assume $G=G(n,r)$ and $H$ to be the generalized Young subgroup, $H= G(a_1,r) \times G(a_2,r)\times\dots\times (a_r,r)$ with $a_1+a_2+\dots+a_r =n$. The derived quotient ${G(n,r)}_{ab}$ is isomorphic to $\mathbb{Z}_2 \times \mathbb{Z}_r$ (which is $\mathbb{Z}_{2r}$, if $r$ is odd) and is generated by $\tau_n$ and $e_n$, where $e_n$ is any vector in $\mathbb{Z}_r^n$ with $\varepsilon(e_n)= \zeta$. We can identify $H_{ab}$ with ${G(a_1,r)}_{ab} \times {G(a_2,r)}_{ab}\times\dots\times {G(a_r,r)}_{ab}.$ We will now describe explicitly the $ver$ map when $G=G(n,r)$ in the following proposition.
\begin{proposition}\label{p7}
The map 
\begin{center}
$ver= ver _{G/H} : {G(n,r)}_{ab}\rightarrow {G(a_1,r)}_{ab} \times {G(a_2,r)}_{ab}\times\dots\times {G(a_r,r)}_{ab}$
\end{center}
is given by
\begin{center}
$ver (\tau_n) = \Big(\tau_{a_1} ^ {{{n-2}\choose{{a_1-2},{a_2},\dots,{a_r}}}},\tau_{a_2} ^ {{{n-2}\choose{{a_1},{a_2-2},\dots,{a_r}}}},\dots,\tau_{a_r}^{{{n-2}\choose{{a_1},{a_2},\dots,{a_r-2}}}}\Big)$
\end{center}
and 
\begin{center}
$ver (e_n) = \Big(e_{a_1} ^ {{n-1}\choose{{a_1-1},{a_2},\dots,{a_r}}},e_{a_2} ^ {{n-1}\choose{{a_1},{a_2-1},\dots,{a_r}}},\dots,e_{a_r}^{{n-1}\choose{{a_1},{a_2},\dots,{a_r-1}}}\Big)$.
\end{center} 
\end{proposition}
\begin{proof}
By considering the following inclusion
\begin{equation*}
S_n/S_{a_1} \times S_{a_2} \times\dots\times S_{a_r} \hooklongrightarrow G(n,r)/G(a_1,r) \times G(a_2,r)\times\dots\times G(a_r,r)
\end{equation*}
and by the following,
 \begin{align*}
 \left|\dfrac{G(n,r)}{G(a_1,r) \times G(a_2,r)\times\dots\times (a_r,r)}\right| & = \dfrac{r^n n!}{r^{a_1} {a_1}! r^{a_2} {a_2}!\dots r^{a_r} {a_r}!}\\
 & = \dfrac{n!}{{a_1}!  {a_2}!\dots {a_r}!}\\
 &=\left|\dfrac{S_n}{S_{a_1} \times S_{a_2} \times\dots\times S_{a_r}}\right|,
 \end{align*}
we may identify the quotient $$G(n,r)/{G(a_1,r) \times G(a_2,r)\times\dots\times G(a_r,r)}$$ with $${S_n/S_{a_1} \times S_{a_2} \times\dots\times S_{a_r}}$$ and we use a transversal of ${S_n/S_{a_1} \times S_{a_2} \times\dots\times S_{a_r}}\rightarrow S_n$ to form a transversal 
 \begin{equation*}
 t: G(n,r)/{G(a_1,r) \times G(a_2,r)\times\dots\times G(a_r,r)} \rightarrow G(n,r)
 \end{equation*}
 whose image lies in $S_n$.
 \begin{center}
\begin{tikzcd}
(S_n)_{ab} \arrow[r, "ver"]\arrow[d]& (S_{a_1})_{ab}\times (S_{a_2})_{ab} \times\dots\times (S_{a_r})_{ab}\arrow[d] 
\\(G(n,r))_{ab}\arrow[r, "ver" ]& (G(a_1,r))_{ab} \times (G(a_2,r))_{ab}\times\dots\times (G(a_r,r))_{ab}
\end{tikzcd}
 \end{center}
 This gives the formula for $ver (\tau_n) $.

In line with the notations from the previous proof , we have
 \begin{equation*}
{^{{t(x)}^{-1}}x  }= J.
 \end{equation*}
 Hence, ${^{{t(x)}^{-1}}}(1)\in J_{k}$ if and only if $1\in x_{k}$ for $1\leq k\leq r$ and
 \begin{equation*}
 h_{x,e_1} = {t(x)}^{-1} e_1 t(x) = e_k, \text{ where }{^{{t(x)}^{-1}}}(1)=k.
 \end{equation*}
 Therefore, modulo $D(H)$, one may write 
 \begin{equation*}
 h_{x,e_1} = {t(x)}^{-1} e_1 t(x) = e_k, \text{ if } 1\in x_{k} \text{ for }1\leq k\leq r,
 \end{equation*}
there are ${{n-1}\choose{{a_1},{a_2},\dots,{a_k-1},\dots,{a_r}}}$ elements of $\mathscr{P}(J_n) $ containing $1$ in its $k$th set. Hence,
\begin{equation*}
ver (e_n) = \Big(e_{a_1} ^ {{n-1}\choose{{a_1-1},{a_2},\dots,{a_r}}},e_{a_2} ^ {{n-1}\choose{{a_1},{a_2-1},\dots,{a_r}}},\dots,e_{a_r}^{{n-1}\choose{{a_1},{a_2},\dots,{a_r-1}}}\Big).
\end{equation*}
\end{proof}
For $i_1,i_2,\dots,i_r$ be the non-negative integers, consider the following
\begin{align*}
&det\Big(\rho_{\lambda_1}^0 (\tau_{a_1} ^{i_1}) \boxtimes \rho_{\lambda_2}^1 (\tau_{a_2} ^{i_2}) \boxtimes\dots \boxtimes  \rho_{\lambda_r}^{r-1} (\tau_{a_r}^{i_r})\Big)\\
&=  det\Big({\rho_{\lambda_1}^0 (\tau_{a_1}})^{i_1}\Big)^{f_{\lambda_2}f_{\lambda_3}\dots f_{\lambda_r}}\dots {\ }det\Big({\rho_{\lambda_r}^{r-1} (\tau_{a_r}} )^{i_r}\Big)^{f_{\lambda_1}f_{\lambda_2}\dots f_{\lambda_{r-1}}}\\
&=  det\Big({\rho_{\lambda_1}^0 (\tau_{a_1}})^{i_1}\Big)^{\hat{f}_{\lambda_1}}\dots {\ }det\Big({\rho_{\lambda_r}^{r-1} (\tau_{a_r}} )^{i_r}\Big)^{\hat{f}_{\lambda_r}}\\
&= {(-1)}^{g_{\lambda_1}\hat{f}_{\lambda_1}i_1 } {(-1)}^{g_{\lambda_2}\hat{f}_{\lambda_2}i_2} \dots {(-1)}^{g_{\lambda_r}\hat{f}_{\lambda_r}i_r}\\
&= {(-1)}^{\sum\limits_{k=1}^r{ g_{\lambda_k}\hat{f}_{\lambda_k}i_k}}
\end{align*}
where $\hat{f}_{\lambda_i}=f_{\lambda_1}f_{\lambda_2}\dots f_{\lambda_{i-1}}f_{\lambda_{i+1}}\dots f_{\lambda_r}.$

Hence,
\begin{equation*}
det\Big(\rho_{\lambda_1}^0  \boxtimes {\rho_{\lambda_2}}^1\boxtimes\dots\boxtimes  \rho_{\lambda_r}^{r-1} \Big)\Big(ver (s_1)\Big) = {(-1)}^{\sum\limits_{k=1}^r{ g_{\lambda_k}\hat{f}_{\lambda_k}{{{n-2}\choose{{a_1},\dots,{a_k-2},\dots,{a_r}}}}}}.
\end{equation*}
One can also compute,
\begin{align*}
&det\Big(\rho_{\lambda_1}^0 (e_{a_1} ^{i_1}) \boxtimes \rho_{\lambda_2}^1 (e_{a_2} ^{i_2}) \boxtimes \dots\boxtimes  \rho_{\lambda_r}^{r-1} (e_{a_r}^{i_r})\Big)\\
&=  det\Big({\rho_{\lambda_1}^0 (e_{a_1}})^{i_1}\Big)^{\hat{f}_{\lambda_1}} det\Big({\rho_{\lambda_2}^1 (e_{a_2}})^{i_2}\Big)^{\hat{f}_{\lambda_2}}\dots{\ }det\Big({\rho_{\lambda_r}^{r-1} (e_{a_r}} )^{i_r}\Big)^{\hat{f}_{\lambda_r}}\\
&=1. det\Big((\zeta {id}_{f_{\lambda_2}})^{i_2}\Big)^{\hat{f}_{\lambda_2}}\dots{\ } det\Big((\zeta^{r-1} {id}_{f_{\lambda_r}})^{i_r}\Big)^{\hat{f}_{\lambda_r}}\\
&= \zeta^{f_{\lambda_1}f_{\lambda_2}\dots f_{\lambda_r}i_2} (\zeta^2)^{f_{\lambda_1}f_{\lambda_2}\dots f_{\lambda_r}i_3}\dots (\zeta^{r-1})^{f_{\lambda_1}f_{\lambda_2}\dots f_{\lambda_r}i_r}\\
&= \zeta^{\Big(\sum\limits_{k=1}^{r-1} ki_{k+1}\Big)f_{\lambda_1}f_{\lambda_2}\dots f_{\lambda_r}}.
\end{align*}
Hence, 
\begin{equation*}
det\Big(\rho_{\lambda_1}^0  \boxtimes {\rho_{\lambda_2}}^1\boxtimes\dots \boxtimes  \rho_{\lambda_r}^{r-1} \Big)\Big(ver (e_n)\Big)=  \zeta^{\sum\limits_{k=1}^{r-1} k{{n-1}\choose{{a_1},\dots,{a_{k+1}-1},\dots,{a_r}}}f_{\lambda_1}f_{\lambda_2}\dots f_{\lambda_r}}. 
\end{equation*}
Since the permutation module $\mathbb{C}[G/H]$ coming from the action of $G$ on the set $\mathscr{P}(J_n) $  factors through the action of $S_n$,  $e_n$ acts trivially and $s_1$ acts by permuting the multisets. The number of orbits of $s_1$ on this set equals to the number of sets of $\rho_\lambda(J_n)$, which contain $1$ and $2$ in two different sets in the multiset, which equals  ${\dfrac{(n-2)!}{a_1!a_2!\dots a_r!}}\sum\limits_{i \neq j}{a_ia_j}$. This gives
\begin{align*}
det( R_{G/H}) (e_n)&=1   \text{ and}\\
det( R_{G/H}) (s_1)&=(-1)^{{\dfrac{(n-2)!}{a_1!a_2!\dots a_r!}} \sum\limits_{i \neq j}{a_ia_j}}
\end{align*}
so that 
\begin{equation*}
det( R_{G/H})=(sgn^o)^{{n-2}\choose{{a_1-1},{a_2-1},\dots,{a_r-1}}}.
\end{equation*}
For $\lambda=(\lambda_1, \lambda_2,\dots,\lambda_r)\in P(n,r)$ we define 
\begin{equation}\label{xlambda}
x_{\lambda} =f_{\lambda_1}f_{\lambda_2}\dots f_{\lambda_r} \sum\limits_{k=1}^{r-1} k{{n-1}\choose{{a_1},\dots,{a_{k+1}-1},\dots,{a_r}}} \in \mathbb{Z}_r 
\end{equation}
and
\begin{equation}\label{ylambda}
y_{\lambda} =f_{\lambda_1}f_{\lambda_2}\dots f_{\lambda_r} {\dfrac{(n-2)!}{a_1!a_2!\dots a_r!}}\sum\limits_{i \neq j}{a_ia_j} +\sum\limits_{k=1}^r{ g_{\lambda_k}\widehat{f}_{\lambda_k}{{n-2}\choose{{a_1},\dots,{a_k-2},\dots,{a_r}}}} \in \mathbb{Z}_2.
\end{equation}
We are now ready to describe the determinant of the irreducible representation $\rho_{\lambda}$ of $G(n,r)$ as follows:
\begin{theorem}\label{t4}
For a multipartition $\lambda \in P(n,r)$, we have
\begin{equation}
det \rho _{\lambda} = \zeta^{x_{\lambda}} sgn^{y_{\lambda}}.
\end{equation}
\end{theorem}

\section{characterization of the determinant in terms of $({a_1},\dots,{a_r}) $}
The determinant formula derived in the last section shows that the question of computing the determinant of an irreducible representation $\lambda$ of $G(n,r)$ relays only on computing the parity of $x_{\lambda}$ and $y_{\lambda}$. Although the determinant involves the multipartitions, one checks without difficulty that it merely uses the underlying composition $a$. In this section, we translate the question of characterizing the determinant of all multipartitions to that of questions about the characterization of the underlying composition irrespective of the parts of the given multipartition.
\begin{theorem}\label{tx}
Given any composition $a = (a_1,\dots,a_r)\models n$ with $r$ odd prime, such that the multinomial coefficients ${{n-1}\choose{{a_1},\dots,{a_{k}-1},\dots,{a_r}}}$ lie in the same residue class modulo $r$, for  $1\leq k\leq r$, then $x_\lambda=0{\ } \mod{\ } r$.
\end{theorem} 
\begin{proof}
If all the multinomial coefficients ${{n-1}\choose{{a_1},\dots,{a_{k}-1},\dots,{a_r}}}$ lie in the same residue class modulo $r$, we have 
\begin{equation*}
\sum_{k=1}^{r-1}{k{{n-1}\choose{{a_1},\dots,{a_{k+1}-1},\dots,{a_r}}}}  =\dfrac{(r-1)r}{2} {{n-1}\choose{{a_1},\dots,{a_{k+1}-1},\dots,{a_r}}} {\ }mod{\ } r.
\end{equation*}
Therefore, $x_\lambda=0{\ } \mod{\ } r$.
\end{proof}
\begin{theorem}\label{ty}
If $\lambda\in P(n,r)$ with $({a_1},\dots,{a_r})$ being the underlying composition of $n$ and satisfies any of the following conditions:
\begin{enumerate}
\item there exists two $a_i ,a_j$, $i\neq j$ such that $ (bin(a_i)\setminus ord(a_i)) \cap (bin(a_j)\setminus ord(a_j)) \neq \phi $
\item there exists $3$ $a_k$'s which are  $2 {\ } mod {\ }4$ 
\item there exists $3$ $a_k$'s congruent to $3, 2, 1$ mod $4$ respectively 
\item atleast $4$ of the $a_k$'s are odd, 
\end{enumerate}
then $y_{\lambda}$ is even.
\end{theorem}
\begin{proof}
We prove by showing that ${{n-2}\choose{{a_1},\dots,{a_i -1},\dots,{a_j -1},\dots,{a_r}} }$ and ${{n-2}\choose{{a_1},\dots,{a_k -2},\dots,{a_r}} }$ are even for $1\leq i \neq j,\,k \leq r$. Observe that, ${{n-2}\choose{{a_1},\dots,{a_i -1},\dots,{a_j -1},\dots,{a_r}} }$ can be written as the product of any of the following binomial coefficients:
\begin{center}
${{a_1+a_2-1}\choose{{a_1-1},{a_2}} }$ or ${{a_1+a_2-2}\choose{{a_1-1},{a_2-1}} }$ or ${{a_1+a_2}\choose{{a_1},{a_2}} }$ or ${{a_1+a_2+a_3-2}\choose{{a_1+a_2-1},{a_3}} }$ or ${{a_1+a_2+a_3-2}\choose{{a_1+a_2-1},{a_3-1}} }$
\end{center}
Similarly, ${{n-2}\choose{{a_1},\dots,{a_k -2},\dots,{a_r}} }$ can be wriiten as a product of ${{a_1+a_2-2}\choose{{a_1-2},{a_2}} }$ or ${{a_1+a_2}\choose{{a_1},{a_2}} }$. 

As an immediate consequence of using Lemma \ref{l1}, the following holds
\begin{enumerate}
\item ${{a_1+a_2-1}\choose{{a_1-1},{a_2}} }$, ${{a_1+a_2-2}\choose{{a_1-1},{a_2-1}} }$, ${{a_1+a_2}\choose{{a_1},{a_2}} }$ and ${{a_1+a_2-2}\choose{{a_1-2},{a_2}} }$ are even if ${a_1},{a_2}$ have a $1$ in the same position, after their first non-zero digit from the right, in their binary expansion
\item ${{a_1+a_2-2}\choose{{a_1-1},{a_2-1}} }$, ${{a_1+a_2}\choose{{a_1},{a_2}} }$ are even if  both ${a_1}$ and ${a_2}$ are congruent to $2 {\ } mod {\ }4$ 
\item ${{a_1+a_2+a_3-2}\choose{{a_1+a_2-1},{a_3}} }$, ${{a_1+a_2+a_3-2}\choose{{a_1+a_2-1},{a_3-1}} }$, ${{a_1+a_2+a_3-2}\choose{{a_1-2},{a_2+a_3-1}} }$, ${{a_1+a_3}\choose{{a_1},{a_3}} }$ and ${{a_1+a_2}\choose{{a_1},{a_2}} }$ are even if $a_1,a_2,a_3$ are congruent to $3, 2, 1$ respectively modulo $4$
\item  ${{a_1+a_2-2}\choose{{a_1-2},{a_2}} }$ and ${{a_1+a_2}\choose{{a_1},{a_2}} }$  are even if both ${a_1},{a_2}$ are odd
\end{enumerate}
 and hence the theorem follows.
\end{proof}
The above two theorems give a characterization of the determinant of $\lambda$ in terms of $({a_1},\dots,{a_r}) $, so that the corresponding irreducible representation has the determinant equal to $\pm 1$ and $\zeta ^s$ for $1\leq s \leq r$ respectively. Finally, by combining the different pieces of above results, we have the following table that gives the possible values of the determinant of some special compositions of $n$ upto the permutation of its parts.
\begin{table}[H]
\caption{Determinant of some special compositions of $n$}
\label{table1}
\resizebox{\textwidth}{!}{
\begin{tabular}{|c|c|c|}
\hline
{\ } & $\mathbf{(a_1,\dots,a_r) \in C(n,r)}$ & \textbf{Possible values of the determinant} \\ \hline
1 & $a_i =a_j\geq 2 $ , $1\leq i,j \leq r$, $r$ odd prime        & 1          \\ \hline
2 & $a_i =a_j {\ } mod {\ } r$, $1\leq i,j \leq r$, $r$ odd prime    & $1,{\ }-1$    \\ \hline
3 & $a_i =0 \mod r$ for some $i$ and $a_k = s \mod r$, $s > \lceil{\frac{r}{2}}\rceil$ for all $k \neq i$, $r$ odd prime      & $1,{\ }-1$                     \\ \hline
4 & $a_i =a_j >2$, for some $i \neq j$        & $\zeta^s$,   $1\leq s \leq r$      \\ \hline
5 & any 4 are odd        & $\zeta^s$,   $1\leq s \leq r$                \\ \hline
6 & any 3 are congruent to $2 {\ } mod {\ }4$       & $\zeta^s$,   $1\leq s \leq r$       \\ \hline
7 & any 2 are congruent to $3 {\ } mod {\ }4$       & $\zeta^s$,   $1\leq s \leq r$       \\ \hline
8 & there exists $3$ $a_k$'s congruent to $3, 2, 1$ mod $4$ respectively 
   & $\zeta^s$,   $1\leq s \leq r$       \\ \hline

\end{tabular}}
\end{table}
\section{Inequalities}
In this section, we present few results on $x_{\lambda}$ and $y_{\lambda}$ which has many interesting consequences in proving the main result of this paper. 
To begin with, for given $\lambda\in P(n,r)$, we have an action of $S_r$ on it by permuting the parts in the multipartition, that is, $\pi(\lambda)= (\lambda_{\pi(1)},\lambda_{\pi(2)},\dots,\lambda_{\pi(r)})$. Let we denote $\sigma( \lambda)=(\lambda_1',\lambda_2',\dots,\lambda_r')$ where $\lambda_k'$ denotes the conjugate of the partition $\lambda_k$. We observe some properties of $x_{\lambda}$ and $y_{\lambda}$ in the next few propositions.
\begin{proposition}\label{p9}
For $\lambda = (\lambda_1,\lambda_2,\dots,\lambda_r)\in P(n,r)$ with $\lambda_k \vdash a_k$ and $1\leq k\leq r$, we have
\begin{enumerate}
\item $x_{\sigma(\lambda)} = x_{\lambda}$
\item $y_{\sigma(\lambda)} = y_{\lambda} + f_{\lambda}{\ } mod {\ } 2.$
\end{enumerate}
\end{proposition}
\begin{proof}
A straightforward verification shows that the Part 1 of the proposition follows directly from the definition of $x_{\lambda}$.
By Lemma 3 of \cite{spallone}, we have
\begin{align*}
y_{\sigma(\lambda)} &=f_{\lambda_1}\dots f_{\lambda_r} {\dfrac{(n-2)!}{a_1!\dots a_r!}}\sum\limits_{i \neq j}{a_ia_j} +\sum\limits_{k=1}^r{ (g_{\lambda_k}+f_{\lambda_k})\widehat{f}_{\lambda_k}{{n-2}\choose{{a_1},\dots,{a_k-2},\dots,{a_r}}}}\\
&= y_{\lambda} + f_{\lambda_1}\dots f_{\lambda_r}\sum\limits_{k=1}^r {{n-2}\choose{{a_1},\dots,{a_k-2},\dots,{a_r}}}.
\end{align*}
Now,
\begin{align*}
f_{\lambda} &= f_{\lambda_1}\dots f_{\lambda_r}  {n \choose {a_1,\dots,a_r}}\\
&= f_{\lambda_1}\dots f_{\lambda_r} {{n-2}\choose{{a_1-1},{a_2-1},a_3,\dots,a_r}}\left(\dfrac{n(n-1)}{a_1a_2}\right)\\
&= f_{\lambda_1}\dots f_{\lambda_r} {{n-2}\choose{{a_1-1},{a_2-1},a_3,\dots,a_r}}\left(\dfrac{(a_1+\dots+a_r)^2-(a_1+\dots +a_2)}{a_1a_2}\right)\\
&= f_{\lambda_1}\dots f_{\lambda_r} {{n-2}\choose{{a_1-1},{a_2-1},a_3,\dots ,a_r}}\left(\dfrac{({a_1}^2+\dots+{a_r}^2)-(a_1+\dots +a_2)}{a_1a_2} +2\dfrac{\sum\limits_{i \neq j}{a_ia_j}}{a_1a_2}\right)\\
&=  f_{\lambda_1}\dots f_{\lambda_r} {{n-2}\choose{{a_1-1},{a_2-1},a_3,\dots,a_r}}\left(\dfrac{({a_1}^2+\dots+{a_r}^2)-(a_1+\dots+a_2)}{a_1a_2} \right) \mod 2\\
&=  f_{\lambda_1}\dots f_{\lambda_r} {{n-2}\choose{{a_1-1},{a_2-1},a_3,\dots,a_r}}
\left(\dfrac{{a_1}(a_1-1)+\dots+{a_r}(a_r-1))}{a_1a_2} \right) \mod 2\\
&= f_{\lambda_1}\dots f_{\lambda_r} \sum\limits_{k=1}^r{{n-2}\choose{{a_1},\dots,{a_k-2},\dots,{a_r}}} \mod 2.
\end{align*}
\end{proof}
The next proposition discusses the action of $S_r$ on the terms $x_{\lambda}$ and $y_{\lambda}$. In particular, $y_{\lambda}$ is invariant under the action of $S_r$.
\begin{proposition} \label{p10}
Suppose $\pi$ is a transposition $(i,i+1)$, then
\begin{enumerate}
\item $x_{\pi(\lambda)} = x_{\lambda} +\dfrac{(a_i-a_{i+1})}{n}f_{\lambda}$
\item $y_{\pi(\lambda)} = y_{\lambda}. $
\end{enumerate}
\end{proposition}
\begin{proof}
Observe that, 
\begin{equation*}
x_{\lambda} = f_{\lambda_1}f_{\lambda_2}\dots f_{\lambda_r}\dfrac{(n-1)!}{a_1!\dots a_r!}[a_2 + 2a_3+\dots+(r-1)a_r].
\end{equation*}
Then
\begin{equation*}
x_{\pi(\lambda)} = f_{\lambda_1}f_{\lambda_2}\dots f_{\lambda_r}\dfrac{(n-1)!}{a_1!\dots a_r!}[a_2 + 2a_3+\dots+(i-1)a_{i+1}+ia_i+\dots+(r-1)a_r]
\end{equation*}
which implies that
\begin{equation*}
x_{\pi(\lambda)} = x_{\lambda} +f_{\lambda_1}f_{\lambda_2}\dots f_{\lambda_r}\dfrac{(n-1)!}{a_1!\dots a_r!}(a_i-a_{i+1} ).
\end{equation*}
One checks easily that the Part (2) of the proposition follows from the fact that both the terms in $ y_{\lambda}$ are invariant under the action of $S_r$.
\end{proof}
As a consequence of the above result, we have the following:
\begin{corollary}\label{c1}
For any $\lambda \in P(n,r)$ and $\pi \in S_r$, we have
\begin{enumerate}
\item $x_{\pi(\lambda)} = x_{\lambda} +\dfrac{\sum_{j=1}^{k}(a_{i_j}-a_{i_j+1} )}{n}f_{\lambda}$, where $k$ is the length of the reduced expression of $\pi = s_{i_1}\dots s_{i_k}$, $s_{i_j} = (i_j,i_{j+1})$ and $i_1,\dots ,i_k \in \{ 1,\dots,r\}$
\item $y_{\pi(\lambda)} = y_{\lambda}. $
\end{enumerate}
\end{corollary}
For the rest of this section unless otherwise stated, we may assume that $r$ is an odd prime.
\begin{lemma}\label{l2}
For a given $a=(a_1,\dots,a_r)\models n$ with all the multinomial coefficients, ${{n-1}\choose{{a_1},\dots,{a_{k}-1},\dots,{a_r}}}$ not from the same residue class modulo $r$, there exists an ordering of $r$ numbers $a_1,\dots,a_r$ such that
\begin{equation*}
\sum\limits_{k=1}^{r-1} k{{n-1}\choose{{a_1},\dots,{a_{k+1}-1},\dots,{a_r}}} \neq 0 \mod r.
\end{equation*} 
\begin{proof}
Pick the largest possible set of integers  $a_k$ from the given composition $a$ such that the multinomial coefficients lies in distinct residue class modulo $r$. Place these elements in the $i$th position so that, $ (i-1){{n-1}\choose{{a_1},\dots,{a_{i}-1},\dots,{a_r}}} = 1 {\ }mod{\ } r$. Place the remaining elements in the left over positions randomly. If $\sum\limits_{k=1}^{r-1} k{{n-1}\choose{{a_1},\dots,{a_{k+1}-1},\dots,{a_r}}} = 0 \mod r$ replace the element in the $1$st position with any of the element placed in the first step.
\end{proof}
\end{lemma}  
\begin{proposition}\label{p10}
Suppose given any  $a=(a_1,\dots,a_r)\models n$ with all the multinomial coefficients, ${{n-1}\choose{{a_1},\dots,{a_{k}-1},\dots,{a_r}}}$ not from the same residue class modulo $r$. Then for each $1\leq s < r$, there exists a permutation $\pi \in S_r$ such that
\begin{equation*}
\sum\limits_{k=1}^{r-1} k{{n-1}\choose{{a_{\pi(1)}},\dots,{a_{\pi(k+1)}-1},\dots,{a_{\pi(r)}}}} = s \mod r.
\end{equation*}
\end{proposition}
\begin{proof}
By Lemma \ref{l2}, w.l.o.g we may assume that $\sum\limits_{k=1}^{r-1} k{{n-1}\choose{{a_1},\dots,{a_{k+1}-1},\dots,{a_r}}} = 1 \mod r$. we have
\begin{align*}
[a_2+2a_3+\dots+(r-1)a_r]\dfrac{(n-1)!}{{a_1!}{a_2!}\dots{a_r!}} &= rt+1 \hspace{1cm} for{\ } some{\ } t\in \mathbb{Z}\\
a_2 \dfrac{(n-1)!}{{a_1!}{a_2!}\dots{a_r!}} &= rt+1 - [2a_3+\dots+(r-1)a_r]\dfrac{(n-1)!}{{a_1!}{a_2!}\dots{a_r!}}\\
sa_2 \dfrac{(n-1)!}{{a_1!}{a_2!}\dots{a_r!}} &= rt'+s - [2sa_3+\dots+s(r-1)a_r]\dfrac{(n-1)!}{{a_1!}{a_2!}\dots{a_r!}}\\
[sa_2+{(2s)}a_3+\dots+{(s(r-1))}a_r]\dfrac{(n-1)!}{{a_1!}{a_2!}\dots{a_r!}} &= s\mod r
\end{align*}
Here, $\pi =  \left(\begin{smallmatrix} 
1& 2 & 3 & \dots & r\\
1& s+1 & {2s+1} & \dots & {s(r-1)+1}
\end{smallmatrix}\right)$ 
\end{proof}
As an immediate consequence of the above result, we deduce a simple and an important result as follows:
\begin{theorem}\label{t7}
$N_{\zeta ^s}(n)$ (resp. $N_{-\zeta ^s}(n)$) are all equal for $1\leq s < r$.
\end{theorem}
\begin{proof}
Suppose, we have a multipartition $\lambda=(\lambda_1,\dots,\lambda_r)\in P(n,r)$ with determinant other than $\pm 1$, i.e., $x_{\lambda} \neq 0 {\ } mod {\ } r$. Then for any $1\leq s <r$, by Proposition \ref{p10}, we have a permutation $\pi$ such that
\[
det \rho _{\pi(\lambda)} =
\begin{cases}
\zeta ^s \quad if{\ } y _{\lambda} = 0 {\ } mod {\ }2\\
-\zeta ^s \quad if{\ } y _{\lambda} = 1 {\ } mod {\ }2.
\end{cases}
\]
\end{proof}
\begin{theorem}\label{t7a}
If $n<r$, then we have $N_{\zeta ^s}(n)$ (resp. $N_{-\zeta ^s}(n)$) are all equal for $1\leq s \leq r$.
\end{theorem}
\begin{proof}
For $n<r$, using Lemma \ref{l2} and Proposition \ref{p10}, w.l.o.g, we have $a_1,\dots,a_r$ with 
\begin{equation*}
\sum_{k=1}^{r-1}{k{{n-1}\choose{{a_1},\dots,{a_{k+1} -1},\dots,{a_r}} }} = 1 \mod r. 
\end{equation*}
Since $\sum _{k=1}^{r}a_k \neq 0 \mod r$, we can always find some $t \in \{1,\dots,r-1\}$ such that 
\begin{equation*}
t\sum_{k=1}^{r-1}{{{n-1}\choose{{a_1},\dots,{a_{k+1} -1},\dots,{a_r}} }} = -1 \mod r.
\end{equation*}
On adding the above equations,
\begin{equation*}
\sum_{k=1}^{r-1}{(t+k){{n-1}\choose{{a_1},\dots,{a_{k+1} -1},\dots,{a_r}} }} = 0 \mod r.
\end{equation*}
The left hand side is actually $\sum_{\pi ^t(k)=1}^{r-1}{\pi ^t(k){{n-1}\choose{{a_{\pi ^t(1)}},\dots,{a_{\pi ^t(k+1)} -1},\dots,{a_{\pi ^t(r)}}} }}$, where $\pi$ is the cyclic permutation $(1 \, 2\dots r) \in S_r$. Using Corollary \ref{c1} and Theorem \ref{t7} together with the fact that $n<r$, we have $N_{\zeta ^s}(n)$ (resp. $N_{-\zeta ^s}(n)$) are all equal for $1\leq s \leq r$.
\end{proof}
Assume $n$ to be a multiple of $r$. If we consider this particular composition $(a_1,a_2,\dots,a_r)\models n$ with $a_i =-a_j{\ } mod {\ }r$ for some $1 \leq i,j \leq r$ and $a_k = 0{\ } mod {\ }r$ for all $k \neq i,j$ then the summation $\sum\limits_{k=1}^{r-1} k{{n-1}\choose{{a_1},\dots,{a_{k+1}-1},\dots,{a_r}}} \neq 0{\ } mod{\ } r$ upto permutation of the indices $\{1\dots k\}$. So the determinant is equal to $\pm 1$ only if any of the $f_{\lambda_k}$ is congruent to $0$ modulo $r$.
\begin{proposition}
Suppose $n$ is not a multiple of $r$. The following holds for $1 \leq s  <r$,
\begin{equation*}
N_{\zeta ^s}(n) \leq N_{1}(n) {\ } and {\ } N_{-\zeta ^s}(n) \leq N_{-1}(n).
\end{equation*}
\end{proposition}
\begin{proof}
Given any composition $a=(a_1,\dots,a_r) \models n$ with $a_k <r$ and all the multinomial coefficients, ${{n-1}\choose{{a_1},\dots,{a_{k}-1},\dots,{a_r}}}$, $1\leq k \leq r$ not in the same residue class modulo $r$, we can find a cyclic permutation as mentioned in Theorem \ref{t7a} which implies that $N_{\zeta ^s}(a)$ (resp. $N_{-\zeta ^s}(a)$) are all equal, for $1\leq s \leq r$. There may also exists compositions of $n$ with all the multinomial coefficients, ${{n-1}\choose{{a_1},\dots,{a_{k}-1},\dots,{a_r}}}$, $1\leq k \leq r$ lying in the same residue class modulo $r$ and partitions of $a_k$ with dimension as a multiple of $r$. Therefore,
\begin{equation*}
N_{\zeta ^s}(n) \leq N_{1}(n) {\ } and {\ } N_{-\zeta ^s}(n) \leq N_{-1}(n).
\end{equation*} 
\end{proof}
\section{Count in terms of $(a_1,a_2,\dots,a_r)$}
In this section, we give formulas for counting the number of multipartitions  whose corresponding irreducible representation has the given determinant by counting the corresponding $y_{\lambda}$ being odd or even and as a consequence of these results we will prove our main result. As we are counting the number of multipartitions obtained from a given composition, and the fact that a composition obtained by permuting the parts of a given composition will give the same set of multipartitions as the former one, and hence we pick the canonical candidate in the compositions, say the partitions. From now on with a slight abuse of notation we use $a= (a_1,\dots,a_r)\vdash n$ for the sequence $a_1 \geq a_2 \geq\dots \geq a_r\geq 0$ and $\sum_{k=1}^{r} a_k =n$.

For any $a=(a_1,a_2,\dots,a_r)\vdash n$, first we calculate the number of multipartitions $(\lambda_1,\dots,\lambda_r)$ of $n$ obtained from the given partition $a\vdash n$ with $|\lambda_k|=a_k$ such that $y_{\lambda}$ is odd. Note that this also gives a count of the corresponding multipartitions with $y_{\lambda}$ is even.  Since the value of $y_{\lambda}$ remains same under under the action of $S_r$ on the multipartition $\lambda$ [Proposition \ref{p10}], we only need to find the number of multipartitions obtained from $a \vdash n$ such that the corresponding $y_{\lambda}$ is odd  (resp. even).

For $a\in C(n)$, let $A_0(a)$ (resp. $A_1(a)$) denote the number of $\lambda=(\lambda_1,\dots,\lambda_r)\in P(n,r)$ such that $|\lambda_k|=a_k$ for all $k\in\{1,\dots,r\}$ and $y_{\lambda}$ is even (resp. odd). Let $A_0(n)$ (resp. $A_1(n)$) be the number of $\lambda=(\lambda_1,\dots,\lambda_r)\in P(n,r)$ such that $y_\lambda$ is even (resp. odd). Also, we have $A_0(a)+A_1(a)= \prod_{k=1}^{r} p(a_k)$, where $p(a_k)$ denotes the number of partitions of $a_k$.
\begin{lemma}\label{l3} [\cite{spallone}, Corollary 2]
Let $B(n)$ denote the number of Chiral partitions of $S_n$ [\cite{amri},Theorem 1].
\begin{enumerate}
\item $\{\lambda \vdash n | f_\lambda = g_\lambda = 1{\ } mod {\ }2\} = \dfrac{1}{2} A(n)$
\item  $\{\lambda \vdash n | f_\lambda = 1{\ } mod {\ }2, g_\lambda = 0 {\ }mod{\ } 2\} = \dfrac{1}{2} A(n)$
\item  $\{\lambda \vdash n | f_\lambda = 0{\ } mod{\ } 2, g_\lambda = 1{\ } mod {\ }2\} = B(n)-\dfrac{1}{2} A(n)$.
\end{enumerate}
\end{lemma}

\begin{proposition}\label{p8}
Let $(a_1,\dots,a_r)\vdash n$.
\begin{enumerate} 
\item[1.] If ${{n-2}\choose{{a_1},\dots,{a_k-2},\dots,{a_r}}}$ are even for all $k\in\{1,\dots,r\}$, then
\[ A_1(a) =
\begin{cases}
\prod_{k=1}^{r}{A(a_k)} &\quad   if{\ }{\dfrac{(n-2)!}{a_1!a_2!\dots a_r!}}\sum\limits_{i \neq j}{a_ia_j}{\ }\text{is odd}\\
0 &\quad  \text{if}{\ }{\dfrac{(n-2)!}{a_1!a_2!\dots a_r!}}\sum\limits_{i \neq j}{a_ia_j}{\ }\text{is even}.\\
\end{cases}
\]
\item[2.] If $m$ of the multinomial coefficients ${{n-2}\choose{{a_1},\dots,{a_i-2},\dots,{a_r}}}$ are odd, then
\begin{equation*}
A_1(a)=\sum_{i \in I}B(a_i)\prod_{{k=1},{k\neq i}}^{r}A(a_k)-\dfrac{m-1}{2}\prod_{k=1}^{r}{A(a_k)},
\end{equation*}
where $I:=\{i|{{n-2}\choose{{a_1},\dots,{a_i-2},\dots,{a_r}}}\text{ is odd }, 1\leq i\leq r \}$.
\end{enumerate}
\end{proposition}
\begin{proof}
\textit{Case 1}: Suppose ${{n-2}\choose{{a_1},\dots,{a_k-2},\dots,{a_r}}}$ are even for all $k\in\{1,\dots,r\}$.

The parity of $y_{\lambda}$ depends on the parity of $f_{\lambda_k}$'s, $g_{\lambda_k}$'s and $ {\dfrac{(n-2)!}{a_1!a_2!\dots a_r!}}\sum\limits_{i \neq j}{a_ia_j}$. We prove this case by looking for all the possible combinations of $f_{\lambda_k}$'s. Suppose say none of the $f_{\lambda_k}$'s are even, one of the $f_{\lambda_k}$ is even, two of them are even and so on upto $r$ of the $f_{\lambda_k}$'s are even. 

If ${\dfrac{(n-2)!}{a_1!a_2!\dots a_r!}}\sum\limits_{i \neq j}{a_ia_j}$ is even, then $y_{\lambda}$ is always even by the assumption in this case and therefore $A_1(a)$ is equal to $0$.

Now suppose, ${\dfrac{(n-2)!}{a_1!\dots a_r!}}\sum\limits_{i \neq j}{a_ia_j}$ is odd. Since any one of the $f_{\lambda_k}$ is even implies the first term in the summation of equation \ref{ylambda} is even eventually makes $y_{\lambda}$ to be even. So the only possibility is that all the $f_{\lambda_k}$'s are odd, while $g_{\lambda_k}$'s have $2^r$(each could be odd or even) choices. Hence,
\begin{align*}
A_1(a)&=2^{r}\dfrac{1}{2} A(a_1)\dfrac{1}{2} A(a_2)\dots\dfrac{1}{2} A(a_r)\\
&= A(a_1)A(a_2)\dots A(a_r).
\end{align*}
\textit{Case 2}: Suppose we assume that $m$ of the multinomial coefficients ${{n-2}\choose{{a_1},\dots,{a_i-2},\dots,{a_r}}}$ are odd for some $m\in\{1,\dots,r\}$.

Note that if atleast two of the $f_{\lambda_k}$'s are even, then $y_{\lambda}$ is always even irrespective of the parities of other terms in the summation and so the only possible cases are only one of the $f_{\lambda_k}$ is even or all of them are odd.

\textit{Case i:} Suppose that only one of the $f_{\lambda_k}$ is even for some $k$. If $k\notin I$, then both the terms in the summation of equation (\ref{ylambda}) are even and hence this case is not possible.
For that particular $i$'s $\in I$ with  $f_{\lambda_i}$ even, the corresponding $g_{\lambda_i}$ has to be odd, and the number of such partitions will be equal to 
\begin{align*}
&2^{r-1}[B(a_i)-\dfrac{1}{2}A(a_i)]\prod_{{k=1},{i\neq k}}^{r}\dfrac{1}{2^{r-1}}A(a_k) \\
&=B(a_i)\prod_{{k=1},{i\neq k}}^{r}A(a_k)-\dfrac{1}{2}\prod_{k=1}^{r}A(a_k).
\end{align*}
The number of multipartitions of the partition $(a_1,\dots,a_r)$ with exactly one of $\lambda_i$ having $f_{\lambda_i}$ even, $g_{\lambda_i}$ odd for $i \in I$ ( $f_{\lambda_k}$'s even for $i \neq k$) is 
\begin{align*}
&=\sum_{i \in I}\big[B(a_i)\prod_{{k=1},{k\neq i}}^{r}A(a_k)-\dfrac{1}{2}\prod_{k=1}^{r}A(a_k)\big]\\
&=\sum_{i \in I}B(a_i)\prod_{{k=1},{k\neq i}}^{r}A(a_k)-\dfrac{m}{2}\prod_{k=1}^{r}A(a_k).
\end{align*}
\textit{Case ii:}
Consider the case that all the $f_{\lambda_k}$'s are odd, then for $k \notin I$, $g_{\lambda_k}$ could be even or odd and we have $2^{r-m}$ choices. For the rest, one of the $m$ $g_{\lambda}$'s should be fixed as odd or even depending on whether is ${\dfrac{(n-2)!}{a_1!a_2!\dots a_r!}}\sum\limits_{i \neq j}{a_ia_j}$ odd or even. Therefore, we have $2^{m-1}$ choices on $g_{\lambda_i}$'s for $i\in I$.  Then the number of possible partitions are
\begin{equation*}
2^{m-1}2^{r-m} \dfrac{1}{2}A(a_1) \dfrac{1}{2}A(a_2)\dots \dfrac{1}{2}A(a_r)=\dfrac{1}{2}\prod_{k=1}^{r}{A(a_k)}.
\end{equation*}
Summing all the $2^r$ possibilities on the parity of $f_{\lambda_k}$'s, we have
\begin{equation*}
\sum_{i \in I}B(a_i)\prod_{{k=1},{k\neq i}}^{r}A(a_k)-\dfrac{m-1}{2}\prod_{k=1}^{r}{A(a_k)}.
\end{equation*}
\end{proof}
As a consequence of the above proposition, we have the following result.
\begin{corollary}\label{c2}
\begin{equation}
\sum\limits_{s=1}^{r} N_{-\zeta ^s}(a)= Number {\ }of  {\ }ways  {\ }of {\ } permuting {\ } \{a_1,\dots,a_r\} \times A_1(a).
\end{equation}
\begin{equation}
\sum\limits_{s=1}^{r} N_{\zeta ^s}(a)= Number {\ }of  {\ }ways  {\ }of {\ } permuting {\ } \{a_1,\dots,a_r\} \times A_0(a).
\end{equation}
\end{corollary}
 For the rest of this section we assume that $r$ is an odd prime. Now we have an immediate result, using the above formula and Theorem \ref{tx} as follows.
\begin{proposition}\label{p15}
Given any $a=(a_1,\dots,a_r) \vdash n$ such that the multinomial coefficients ${{n-1}\choose{{a_1},\dots,{a_{k}-1},\dots,{a_r}}}$, $1\leq k\leq r$ lie in the same residue class modulo $r$,  then we have for $1\leq s < r$
\begin{equation*}
N_{\zeta ^s}(a)= N_{-\zeta ^s}(a) =0 
\end{equation*}
\begin{equation*}
N_1(a) = Number {\ }of  {\ }ways  {\ }of {\ } permuting {\ } \{a_1,\dots,a_r\} \times A_0(a)
\end{equation*}
\begin{equation*}
N_{-1}(a)= Number {\ }of  {\ }ways  {\ }of {\ } permuting {\ } \{a_1,\dots,a_r\} \times A_1(a).
\end{equation*}
\end{proposition}
From table \ref{table1}, we have the following examples of compositions that  satisfies the above formula :
\begin{enumerate}
\item $a_i =a_j {\ } mod {\ } r$, $1\leq i,j \leq r$, $r$ odd prime  
\item $a_i =0 \mod r$ for some $i$ and $a_k = s \mod r$, where $s > \lfloor{\frac{r}{2}}\rfloor$ for all $k \neq i$, $r$ odd prime    
\end{enumerate}
Hence, from now on we make the assumption that given any composition $a=(a_1,\dots,a_r) $ of $ n$ with all the multinomial coefficients ${{n-1}\choose{{a_1},\dots,{a_{k}-1},\dots,{a_r}}}$, $1\leq k\leq r$, doesnot lie in the same residue class modulo $r$.  
\begin{proposition}
Given any $a=(a_1,\dots,a_r) \vdash n$ where $n$ is not a multiple of $r$, with each $a_k <r$. Then, for $1\leq s \leq r$, we have
\begin{equation*}
N_{\zeta ^s}(a) =\dfrac{(r-1)!}{\prod (no.{\ } of{\ }  repetitions {\ }of {\ }a_k's )!} \times A_0(a) 
\end{equation*}
\begin{equation*}
N_{-\zeta ^s}(a) =\dfrac{(r-1)!}{\prod (no.{\ } of{\ }  repetitions {\ }of {\ } a_k's )!} \times  A_1(a).
\end{equation*}
\end{proposition}
\begin{proof}
Note that $a_k<r$ implies that all partitions of $a_k$ has dimension relatively prime to $r$. So by corollary \ref{c2}, with $1\leq s \leq r$
\begin{align*}
N_{\zeta ^s}(a) &=\dfrac{ no{\ } of {\ } ways{\ }  of {\ }permuting {\ }\{a_1,\dots,a_r\} \times A_0(a)}{r}\\
&=\dfrac{(r-1)!}{\prod (no.{\ } of{\ }  repetitions{\ } of {\ } a_k's )!}  \times A_0(a)
\end{align*}
\begin{align*}
N_{-\zeta ^s}(a) &=\dfrac{ no{\ } of{\ }  ways{\ }  of {\ }permuting {\ }\{a_1,\dots,a_r\} \times A_1(a)}{r}\\
&=\dfrac{(r-1)!}{\prod(no.{\ } of{\ }  repetitions {\ }of {\ } a_k's )!} \times  A_1(a).
\end{align*}
\end{proof}
Note that if $n < r$, then for any $a = (a_1,\dots,a_r) \in C(n)$, the above proposition gives us the following closed formulas for $1\leq s \leq r$
\begin{equation}\label{maineq}
\begin{split}
N_{\zeta ^s}(a) =\dfrac{(r-1)!}{\prod (no.{\ } of{\ }  repetitions {\ }of {\ }a_k's )!} \times A_0(a) \\
N_{-\zeta ^s}(a) =\dfrac{(r-1)!}{\prod (no.{\ } of{\ }  repetitions {\ }of {\ } a_k's )!} \times  A_1(a).
\end{split}
\end{equation}
\begin{proposition}
Given any $a = (a_1,\dots,a_r)\vdash n$,
\[
N_{\zeta ^s}(a) +N_{-\zeta ^s}(a)  =
\begin{cases}
\dfrac{(r-2)!r}{\prod (no.{\ } of{\ }  repetitions {\ }of {\ }a_k's )!}\prod_{k=1}^{r}{m_r(S_{a_k})} \quad if{\ }n=0{\ }mod{\ }r\\
\dfrac{(r-1)!}{\prod (no.{\ } of{\ }  repetitions {\ }of {\ }a_k's )!}\prod_{k=1}^{r}{m_r(S_{a_k})} \quad if{\ }n\neq 0{\ }mod{\ }r\\
\end{cases}
\]
where $1\leq s <r $ and

$N_{1}(a) +N_{-1}(a) $
\[
=
\begin{cases}
\dfrac{r!}{\prod (no.{\ } of{\ }  repetitions {\ }of {\ }a_k's )!}\big{[}\prod_{k=1}^{r}{p({a_k})}-\prod_{k=1}^{r}{m_r(S_{a_k})}\big{]}{\ }if{\ }n=0{\ }mod{\ }r\\
\dfrac{(r-1)!}{\prod (no.{\ } of{\ }  repetitions {\ }of {\ }a_k's )!}\big{[}r \prod_{k=1}^{r}{p({a_k})}-(r-1)\prod_{k=1}^{r}{m_r(S_{a_k})}\big{]}\quad if{\ }n\neq 0{\ }mod{\ }r.\\
\end{cases}
\]
\end{proposition}
\begin{proof}
If all the multinomial coefficients ${{n-1}\choose{{a_1},\dots,{a_{k}-1},\dots,{a_r}}}$ does not lie in the same residue class modulo $r$, by using Lemma \ref{l2} and Theorem \ref{t7} we have an equal number of compositions $(a_1,\dots,a_r)$ with the term $\sum_{k=1}^{r-1}k{{n-1}\choose{{a_1},\dots,{a_{k+1}-1},\dots,{a_r}}}$ visiting each residue class not congruent to $0$ mod $r$. In addition, if $n$ is not a multiple of $r$, then the number of compositions $(a_1,\dots,a_r)$ with the summation visiting each residue class modulo $r$ are equal.

Thus the number of compositions with parts from a given $a \vdash n$ such that\\  
$\sum_{k=1}^{r-1}k{{n-1}\choose{{a_1},\dots,{a_{k+1}-1},\dots,{a_r}}} \neq 0 {\ } mod {\ }r$ is
\[
 =
\begin{cases}
\dfrac{(r-2)!r}{\prod (no.{\ } of{\ }  repetitions {\ }of {\ }a_k's )!} \quad if{\ }n=0{\ }mod{\ }r\\
\dfrac{(r-1)!}{\prod (no.{\ } of{\ }  repetitions {\ }of {\ }a_k's )!} \quad if{\ }n\neq 0{\ }mod{\ }r\\
\end{cases}
\]
and since the number of multipartitions with $\prod_{k=1}^r f_{\lambda_k} \neq 0 {\ }mod{\ }r$ is $\prod_{k=1}^{r}{m_r(S_{a_k})}$ the equation follows immediately.

If $n$ is a multiple of $r$ then $x_{\lambda_1,\dots,\lambda_r} =0 {\ }mod{\ }r$   iff $\prod_{k=1}^r f_{\lambda_k} = 0 {\ }mod{\ }r$. Since the compositions of $n$ are obtained from all the possible permutations of $\{a_1,\dots,a_r\}$ and hence we have the formula
\begin{equation*}
N_{1}(a) +N_{-1}(a) = \dfrac{r!}{\prod (no.{\ } of{\ }  repetitions {\ }of {\ }a_k's )!}\big{[}\prod_{k=1}^{r}{p({a_k})}-\prod_{k=1}^{r}{m_r(S_{a_k})}\big{]}
\end{equation*}
If $n$ is not a multiple of $r$,
\begin{align*}
N_{1}(a) +N_{-1}(a)& = N_{\zeta ^s}(a) +N_{-\zeta ^s}(a)  + \dfrac{r!}{\prod (no.{\ } of{\ }  repetitions {\ }of {\ }a_k's )!}\big{[}\prod_{k=1}^{r}{p({a_k})}-\prod_{k=1}^{r}{m_r(S_{a_k})}\big{]}\\
&= \dfrac{(r-1)!}{\prod (no.{\ } of{\ }  repetitions {\ }of {\ }a_k's )!} \big{[}r \prod_{k=1}^{r}{p({a_k})}-(r-1)\prod_{k=1}^{r}{m_r(S_{a_k})}\big{]}
\end{align*}
\end{proof}
This help us to calculate $N_{\zeta^s}(a)$ for some special compositions of $n$ as mentioned below:
\begin{enumerate}
\item $a_i =a_j >2$, for some $i \neq j$        
\item any 4 are odd                      
\item there exists two $a_i ,a_j$, $i\neq j$ such that $ (bin(a_i)\setminus ord(a_i)) \cap (bin(a_j)\setminus ord(a_j)) \neq \phi $
\item any 2 are congruent to $3 {\ } mod {\ }4$
\item there exists $3$ $a_k$'s which are congruent to $2 {\ } mod {\ }4$ 
\item there exists $3$ $a_k$'s congruent to $3,2,1$ respectively modulo $4$.
\end{enumerate}
Combining all the formulas and the inequalities discussed above, knowing any one of the four values $N_{1}(a)$, $N_{-1}(a)$, $N_{\zeta}(a)$, or $N_{-\zeta}(a)$ will help us to get all the remaining values.

Observe that, $N_{\zeta^s}(n) = \sum_{a\in C(n)}{N_{\zeta^s}(a)}$. This helps us to calculate $N_{\zeta^s}(n)$ for a positive integer $n <r$, r prime using the equation \ref{maineq}. For a positive integer $n$ and a prime $r$, we also have
\begin{equation*}
|P(n,r)| = N_{1}(n)+ N_{-1}(n)+ (r-1) \big [N_{\zeta}(n)+N_{-\zeta}(n)\big].
\end{equation*}
Another interesting problem one can consider is to calculate the number of irreducible representations of $S_n$ with dimension relatively prime to both $2$ and an odd prime $r$. Once we have this count, by proving the results similar to that of in Lemma \ref{l3} with an additional condition on the parity of $f_\lambda$ with respect to $r$ and by using a similar approach to that of Proposition \ref{p8}, one can get the value of $N_\zeta(a)$ for all $n$ and $r$.
\section*{Acknowledgements}
The authors are very grateful to Amritanshu Prasad and Steven Spallone for their help and encouragement. AP is supported by Indian Institute of Science Education and Research Thiruvananthapuram institute fellowship. This research was driven by computer exploration using the open-source mathematical software Sage \cite{sage}.
\section{Appendix A}
In this section, we give an alternate approach to find the determinant of the irreducible representations of $G(n,r)$. We follow the approach as that in \cite{Ghosh}.
Let $H$ be a subgroup of the finite group $G$ and let $g_1,\dots,g_m$ be the  representatives for the distinct left cosets of $H$ in $G$. 
Let $e_1:=(1,0,\dots,0;1_{S_n})$ and $s_1:=(\vec{0};(1,2))$. Using the Frobenius character formula for the induced representations, we have the following proposition: 
\begin{proposition}
For given a $\lambda\in P(n,r)$, we have
\begin{equation*}
\chi_{\lambda} (e_1) = f_{\lambda_1}f_{\lambda_2}\dots f_{\lambda_r} \sum\limits_{k=0}^{r-1} \zeta^k{{n-1}\choose{{a_1},\dots,{a_{k+1}-1},\dots,{a_r}}}.
\end{equation*}
\end{proposition}
\begin{proof}
Recall from section 3 that $J_n=\{1,2,\dots,n\}$ and $\mathscr{P}(J_n)$. Let us fix 
\begin{equation*}
J = \{\{1,2,\dots,a_1\},\{a_1+1,\dots,a_1+a_2\},\dots,\{n-a_r,n-a_{r+1},\dots,n\}\} \in \mathscr{P}(J_n).
\end{equation*}
Since 
\begin{equation*}
Orb(J) \cong S_{a_1+\dots+a_r}/S_{a_1}\times\dots\times S_{a_r}, 
\end{equation*}
and by using the Frobenius character formula, we have
\begin{equation*}
{Ind}_{G(a_1,r)\times G(a_2,r) \times\dots\times G(a_r,r)} ^{G(n,r)}\rho_{\lambda_{1}}^0\boxtimes \dots\boxtimes \rho_{\lambda_{r}}^{r-1}(e_1) = \sum\limits_{i=1}^{n}\rho_{\lambda_{1}}^0\boxtimes \dots\boxtimes \rho_{\lambda_{r}}^{r-1}({g_i}^{-1}e_1g_i).
\end{equation*}
Now let us define the map $g_i$ by
\begin{align*}
e_1 &\mapsto {g_i}^{-1}e_1g_i\\
g_i : \small{\begin{pmatrix}
1\\ \vdots \\ a_1 \\a_1+1 \\ \vdots \\ a_1+a_2\\ \vdots \\n
\end{pmatrix}}
&\mapsto 
\begin{pmatrix}
x_{i_1}^{1}\\ \vdots \\ x_{i_{a_1}}^{1} \\x_{i_1}^{2} \\ \vdots \\ x_{i_{a_2}}^{2}\\ \vdots \\ x_{i_{a_r}}^{r}
\end{pmatrix}
= \begin{pmatrix}
 \\  X_{a_1}\\ \\   \\X_{a_2} \\ \\  \vdots \\ \\  X_{a_r}
\end{pmatrix}
\end{align*}
Thus we have,
\begin{equation*}
\#\{g_i |1\in X_{a_1}\} = \#\{g_i |1\notin X_{a_2}\cup \dots \cup X_{a_r} \} = {{n-1}\choose {a_1-1, a_2,\dots,a_r}}.
\end{equation*}
The number of $g_i$'s such that
\begin{equation*}
 \rho_{\lambda_{1}}^0\boxtimes \dots\boxtimes \rho_{\lambda_{r}}^{r-1}({g_i}^{-1}e_1g_i) = \zeta^k ,\,  0\leq k\leq r
\end{equation*}
is ${{n-1}\choose {a_1,\dots,a_{k+1}-1,\dots,a_r}}$. Hence, 
\begin{equation*}
\chi_{\lambda_1,\lambda_2,\dots,\lambda_r} (e_1) = f_{\lambda_1}f_{\lambda_2}\dots f_{\lambda_r} \sum\limits_{k=0}^{r-1} \zeta^k{{n-1}\choose{{a_1},\dots,{a_{k+1}-1},\dots,{a_r}}}
\end{equation*}
\end{proof}
\begin{proposition}
For a given $\lambda\in P(n,r)$, we have
\begin{equation*}
\chi_{\lambda} (s_1) = \sum\limits_{k=1}^r{\widehat{f}_{\lambda_k} \chi_{\lambda_k}((1,2)){{n-2}\choose{{a_1},\dots,{a_k-2},\dots,{a_r}}}}.
\end{equation*}
\end{proposition}
\begin{proof}
\begin{equation*}
{Ind}_{G(a_1,r)\times G(a_2,r) \times\dots\times G(a_r,r)} ^{G(n,r)}\rho_{\lambda_{1}}^0\boxtimes \dots\boxtimes \rho_{\lambda_{r}}^{r-1}( s_1 )= \sum\limits_{i=1}^{n}\rho_{\lambda_{1}}^0\boxtimes \dots\boxtimes \rho_{\lambda_{r}}^{r-1}({g_i}^{-1}s_1g_i).
\end{equation*} 
Also, we know that ${g_i}^{-1}s_1g_i = {g_i}^{-1}(1){g_i}^{-1}(2)$.  Now consider the map $g_i$ given by
\begin{align*}
s_1 &\mapsto ({g_i}^{-1}s_1g_i)\\
g_i : \begin{pmatrix}
1\\ \vdots \\ a_1 \\a_1+1 \\ \vdots \\ a_1+a_2\\ \vdots \\n
\end{pmatrix}
&\mapsto 
\begin{pmatrix}
x_{i_1}^{1}\\ \vdots \\ x_{i_{a_1}}^{1} \\x_{i_1}^{2} \\ \vdots \\ x_{i_{a_2}}^{2}\\ \vdots \\ x_{i_{a_r}}^{r}
\end{pmatrix}
= \begin{pmatrix}
 \\  X_{a_1}\\ \\   \\X_{a_2} \\ \\  \vdots \\ \\  X_{a_r}
\end{pmatrix}.
\end{align*}
Thus, 
\begin{equation*}
\#\{g_i |1\in X_{a_i} and {\ }2\in  X_{a_i} \} = {{n-2}\choose {a_1,\dots, a_i-2,\dots,a_r}}.
\end{equation*}
Hence, 
\begin{equation*}
\chi_{\lambda} (s_1) = \sum\limits_{k=1}^r{\widehat{f}_{\lambda_k} \chi_{\lambda_k}(s_1){{n-2}\choose{{a_1},\dots,{a_k-2},\dots,{a_r}}}}.
\end{equation*}
\end{proof}
This character formulas help us to calculate the determinant of $\rho_{\lambda}$. Note that the determinant (resp. trace) of a matrix is the product (resp. sum) of all eigen values.
Let $x_{\lambda}$ denote the power of  $\zeta$ in $det \rho_{\lambda}$. From the above formula for $\chi_{\lambda} (s_1)$, we can see that each eigen value $\zeta^k$ for $0 \leq k \leq r-1$  has multiplicity $f_{\lambda_1}f_{\lambda_2}\dots f_{\lambda_r} {{n-1}\choose{{a_1},\dots,{a_{k+1}-1},\dots,{a_r}}}$. Hence, we have
\begin{equation*}
x_{\lambda} =f_{\lambda_1}f_{\lambda_2}\dots f_{\lambda_r} \sum\limits_{k=1}^{r-1} k{{n-1}\choose{{a_1},\dots,{a_{k+1}-1},\dots,{a_r}}}.
\end{equation*}
Let $y_{\lambda}$ denote the multiplicity of  $-1$ an eigenvalues of $\rho_{\lambda}$. Then we have,
\begin{align*}
y_{\lambda} &=\dfrac{\chi_{\lambda}(1,1,\dots,1;1_{S_n})-\chi_{\lambda}(s_1)}{2}\\
&= \dfrac{f_{\lambda_1}\dots f_{\lambda_r}  {n \choose {a_1,\dots,a_r}}-\sum\limits_{k=1}^r{\widehat{f}_{\lambda_k} \chi_{\lambda_k}(s_1){{n-2}\choose{{a_1},\dots,{a_k-2},\dots,{a_r}}}} }{2}.
\end{align*}
Using $g_{\lambda} = \dfrac{f_{\lambda}-\chi_{\lambda}}{2}$, we have
\begin{align*}
y_{\lambda} &=\frac{f_{\lambda_1}\dots f_{\lambda_r}}{2}  [{n \choose {a_1,\dots,a_r}}- \sum\limits_{k=1}^r{ {{n-2}\choose{{a_1},\dots,{a_k-2},\dots,{a_r}}}}] +\sum\limits_{k=1}^r{\widehat{f}_{\lambda_k} g_{\lambda_k}{{n-2}\choose{{a_1},\dots,{a_k-2},\dots,{a_r}}}}\\
&=f_{\lambda_1}f_{\lambda_2}\dots f_{\lambda_r} {\dfrac{(n-2)!}{a_1!a_2!\dots a_r!}}\sum\limits_{i \neq j}{a_ia_j} +\sum\limits_{k=1}^r{ g_{\lambda_k}\widehat{f}_{\lambda_k}{{n-2}\choose{{a_1},\dots,{a_k-2},\dots,{a_r}}}}. 
\end{align*}
Therefore, \begin{equation*}
det \rho _{\lambda} = \zeta^{x_{\lambda}} (sgn)^{y_{\lambda}}.
\end{equation*}
\section{Appendix B}
In this section, for given small values of $r$, a prime, we have computed the table values of $N_{\zeta^s}(n)$ and logplot, base $2$, of $N_{\zeta^s}(n)$ for $n$ varying from $1$ to $10$. The red and blue lines are $N_{1}(n)$ and $N_{-1}(n)$ respectively. The green and orange lines are  $N_{\zeta^s}(n)$ (i.e., $x_\lambda \neq 0{\ } mod {\ }r$ and $y_\lambda = 0{\ } mod {\ }2$) and $N_{-\zeta^s}(n)$ (i.e., $x_\lambda \neq 0{\ } mod {\ }r$ and $y_\lambda = 1{\ } mod {\ }2$), for $1\leq s <r$ respectively.

For $r=2$
\begin{table}[h]
\begin{tabular}{|c|c|c|c|c|}
\hline
n & $N_1(n)$ & $N_{\zeta}(n)$  & $N_{-\zeta}(n)$ &  $N_{-1}(n)$ \\ \hline
1 & 1        & 1        & 0        & 0        \\ \hline
2 & 1        & 1        & 2        & 1        \\ \hline
3 & 2        & 2        & 2        & 4        \\ \hline
4 & 4        & 4        & 8        & 4        \\ \hline
5 & 8        & 4        & 4        & 20       \\ \hline
6 & 33       & 8        & 16       & 8        \\ \hline
7 & 46       & 16       & 16       & 32       \\ \hline
8 & 69       & 28       & 60       & 28       \\ \hline
9 & 116      & 8        & 8        & 168         \\ \hline
10 & 417     & 32       & 16       & 16          \\ \hline
\end{tabular}
\end{table}
\begin{figure}[H]
  \includegraphics[width=\linewidth]{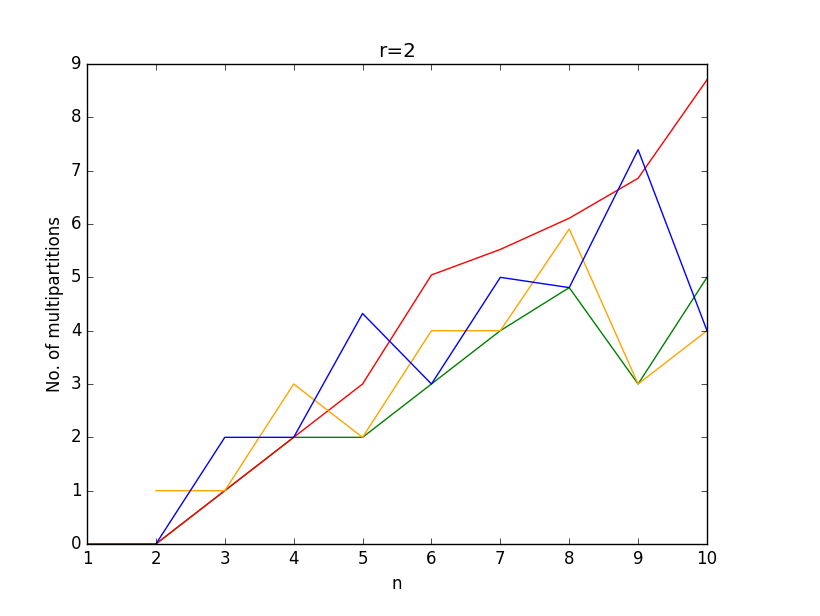}
  \label{fig:r2}
\end{figure}
for $r=3$
\begin{table}[h]
{\tiny
\begin{tabular}{|c|c|c|c|c|c|c|c|c|c|c|}
\hline
n  & $N_1(n)$ & $N_{\zeta}(n)$ & $N_{\zeta^2}(n)$ & $N_{-\zeta}(n)$ &$N_{-\zeta^2}(n)$ & $N_{-1}(n)$ \\ \hline
1  & 1        &1       & 1        &0        & 0 & 0        \\ \hline
2  & 1        & 1        & 1        &2        & 2 & 2        \\ \hline
3  & 1        &4        & 4       & 5        & 5 & 3        \\ \hline
4  & 12       &3        &3      & 6        &  6          & 21       \\ \hline
5  & 21       &9        &9        & 18       & 18 & 33       \\ \hline
6  & 47       & 51       &51       & 18      & 18 & 30       \\ \hline
7  & 201      & 36       & 36       & 18       & 18 & 120      \\ \hline
8  & 244      &97       & 97       & 65      & 65  & 242      \\ \hline
9  & 280      & 217      &217     &197      &197 & 362      \\ \hline
10 & 2454     & 21       & 21      & 6        &6 & 132      \\ \hline
\end{tabular}}
\end{table}
\begin{figure}[H]
  \includegraphics[width=\linewidth]{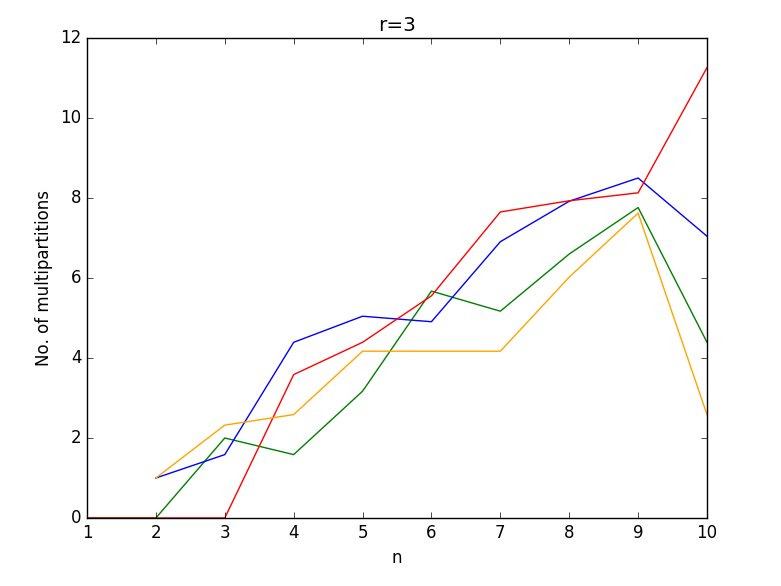}
  \label{fig:r3}
\end{figure}
for r=5
\begin{table}[h]
{\tiny
\begin{tabular}{|c|c|c|c|c|c|c|c|c|c|c|c|c|c|}
\hline
n  & $N_1(n)$ & $N_{\zeta}(n)$ & $N_{\zeta^2}(n)$ & $N_{\zeta^3}(n)$ & $N_{\zeta^4}(n)$ & $N_{-\zeta}(n)$ & $N_{-\zeta^2}(n)$ & $N_{-\zeta^3}(n)$ & $N_{-\zeta^4}(n)$ &$N_{-1}(n)$ \\ \hline
1  &1     &1        & 1        & 1        &1        & 0        & 0 & 0 & 0        & 0        \\ \hline
2  & 1        &1       & 1       & 1        & 1        & 3       & 3       & 3      & 3 &3        \\ \hline
3  & 5     &5       & 5      & 5       & 5        & 8        & 8        & 8        &8 & 8        \\ \hline
4  &11       &11      & 11      & 11       & 11 & 27       & 27       &27       & 27      & 27       \\ \hline
5  &37    &36     & 36    & 36      & 36 & 69      & 69       &69       & 69      & 49       \\ \hline
6  &905     &15     & 15      & 15      & 15    &10       & 10       & 10 &10      & 260      \\ \hline
7  &1926     & 66      & 66       &66      & 66 & 34        & 34        &34       & 34       & 664      \\ \hline
8  &3357   &197     &197      & 197       &197  & 128      & 128 & 128 & 128       & 2168     \\ \hline
9  & 6465     & 540      & 540     & 540     & 540  &410      & 410     &410     &410  & 4460     \\ \hline
10 & 20519    &2477    & 2477     &2477       & 2477    & 93       & 93       & 93 &93 & 228      \\ \hline
\end{tabular}}
\end{table}
\begin{figure}[H]
  \includegraphics[width=\linewidth]{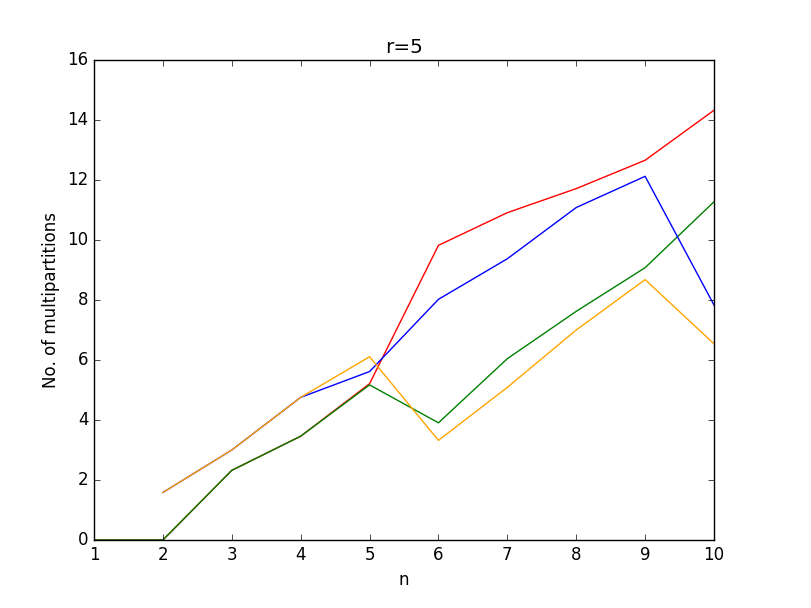}
  \label{fig:r5}
\end{figure}
for $r=7$
\begin{table}[h]
\resizebox{\textwidth}{!}{
\begin{tabular}{|c|c|c|c|c|c|c|c|c|c|c|c|c|c|c|c|c|c|}
\hline
n  & $N_1(n)$ & $N_{\zeta}(n)$ & $N_{\zeta^2}(n)$ & $N_{\zeta^3}(n)$ & $N_{\zeta^4}(n)$ & $N_{\zeta^5}(n)$ & $N_{\zeta^6}(n)$ & $N_{-\zeta}(n)$ & $N_{-\zeta^2}(n)$ & $N_{-\zeta^3}(n)$ & $N_{-\zeta^4}(n)$ & $N_{-\zeta^5}(n)$ & $N_{-\zeta^6}(n)$ &$N_{-1}(n)$ \\ \hline
1   & 1        & 1        & 1        & 1       & 1       & 1       & 1       & 0        & 0        & 0       & 0        & 0        & 0       & 0        \\ \hline
2 &1        & 1      & 1        & 1        &1        & 1        & 1        & 4        & 4        & 4        & 4        & 4        & 4        & 4        \\ \hline
3  & 7        & 7        & 7        & 7        & 7        & 7        &7        & 13       & 13       & 13       & 13       & 13       & 13       &13       \\ \hline
4  & 19       & 19       & 19       & 19       & 19       & 19       & 19       & 51       & 51       & 51       & 51       & 51       & 51       & 51       \\ \hline
5  & 74      & 74       & 74       & 74       &74       & 74      & 74       & 147      & 147      & 147      & 147      & 147      & 147      & 147      \\ \hline
6  & 534     & 534      & 534      & 534     & 534      & 534      & 534     & 112      & 112      & 112      & 112     & 112      & 112      & 112      \\ \hline
7  & 1397     & 1410     & 1410     & 1410     & 1410     & 1410    & 1410     & 368      & 368      & 368      & 368      & 368      & 368      & 340      \\ \hline
8  & 22351    & 21       & 21       & 21       & 21       & 21       & 21      & 28       & 28       & 28       & 28       & 28       & 28       & 9660     \\ \hline
9  & 52199    & 105      & 105      & 105      & 105      & 105      & 105      & 140      & 140      & 140      & 140      & 140      & 140      & 26796    \\ \hline
10 & 185871   & 910      & 910      & 910      & 910      & 910      & 910      & 70       & 70       & 70      & 70       & 70       & 70       & 1148     \\ \hline
\end{tabular}}
\end{table}
\begin{figure}[H]
  \includegraphics[width=\linewidth]{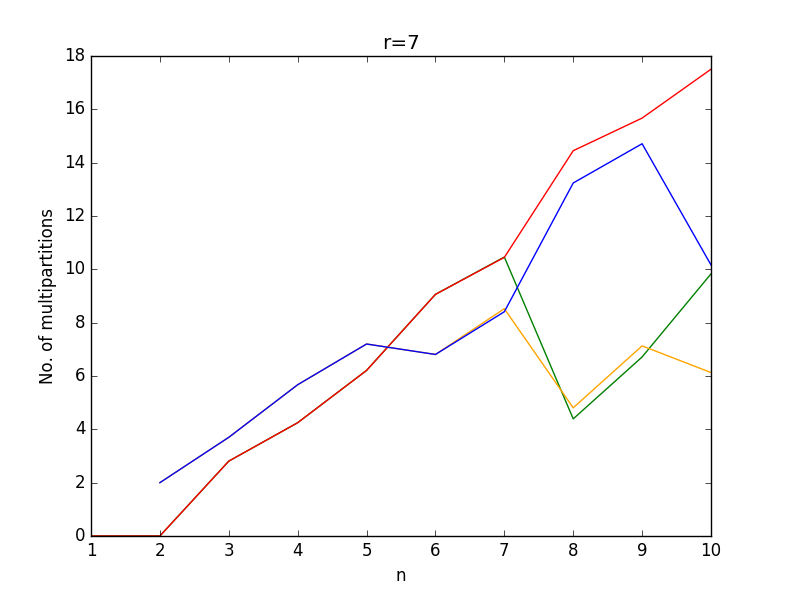}
  \label{fig:r7}
\end{figure}


\begin{thebibliography}{50}
\bibitem {amri}
Arvind Ayyer, Amritanshu Prasad, and Steven Spallone. Representations of symmetric
groups with non-trivial determinant. \textit{J. Combin. Theory Ser. A}, 150 (2017) 208-232.
\bibitem {det}
Colin J. Bushnell and Guy Henniart. The local Langlands conjecture for
GL(2). \textit{Grundlehren der Mathematischen Wissenschaften
[Fundamental Principles of Mathematical Sciences]}. Springer-Verlag,
Berlin, 335 (2006).
\bibitem {ELucas}
Edouard Lucas , Théorie des Fonctions Numériques Simplement Périodiques.\textit{ American Journal of Mathematics}, 1 (2) (1878) 184–196. 
\bibitem {Ghosh}
D. Ghosh, Determinant of representations of hyperoctahedral groups. \textit{ Master's Thesis}, Indian Institute of Science Education and Research Pune, (2017).
\bibitem {spallone}
D. Ghosh and S. Spallone. Determinants of representations of Coxeter groups.
\textit{Journal of Algebraic Combinatorics}, 49 (3) (2019) 229-265.
\bibitem {JamesKerber}
G.D. James, and  A. Kerber, The Representation Theory of the Symmetric Group. \textit{Lecture Notes in Math}, Springer Verlag, 682 (1978). 
\bibitem {Kerber}
A. Kerber, Applied Finite Group Actions, Springer, (1999).
\bibitem {mac}
Ian G. Macdonald. On the degrees of the irreducible representations
of symmetric groups. \textit{Bulletin of the London Mathematical Society},
3(2) (1971) 189-192.
\bibitem {mac_1}
Ian G. Macdonald. On the degrees of the irreducible representations of finite Coxeter groups. Journal of the London Mathematical Society, 2(6) (1973) 298-300.
\bibitem {Olsson}
J. B. Olsson. Combinatorics and Representations of Finite Groups. \textit{Vorlesungen aus dem Fachbereich Mathematik der Universitat GH Essen}, Universitat Essen, 20 (1993) 1993.
\bibitem {sage}
The Sage Developers. Sage Mathematics Software (Version 6.10) (2015). http://www.sagemath.org.
\bibitem {Stanley} R. P. Stanley. Enumerative Combinatorics. Vol 2 \textit{Cambridge University Press}, (2001). 
\end{thebibliography}
\end{document}